\documentclass[runningheads]{llncs}
\usepackage{geometry}
\usepackage{microtype}
\usepackage{graphicx}
\usepackage{subfigure}
\usepackage{booktabs} 

\usepackage{amsmath,amssymb}
\usepackage{color,graphicx}
\usepackage{commath}
\usepackage{lipsum}
\usepackage{mathrsfs}
\usepackage{algorithm}
\usepackage{algorithmic}

\begin{document}
%
\title{Approximating the Optimal Transport Plan via Particle-Evolving Method}
%
%
\author{Shu Liu\inst{1} \and
Haodong Sun\inst{1} \and
Hongyuan Zha\inst{2}}

%
\authorrunning{S. Liu et al.}
%
%
\institute{Georgia Institute of Technology, Atlanta, USA \\
\email{ \{sliu459,hsun310\}@gatech.edu}
\and
 The Chinese University of Hong Kong, Shenzhen, China \\
 \email{zhahy@cuhk.edu.cn}
 }
\maketitle              
\begin{abstract}
Optimal transport (OT) provides powerful tools for comparing probability measures in various types. The Wasserstein distance which arises naturally from the idea of OT is widely used in many machine learning applications. Unfortunately, computing the Wasserstein distance between two continuous probability measures always suffers from heavy computational intractability. In this paper, we propose an innovative algorithm that iteratively evolves a particle system to match the optimal transport plan for two given continuous probability measures. The derivation of the algorithm is based on the construction of the gradient flow of an Entropy Transport Problem which could be naturally understood as a classical Wasserstein optimal transport problem with relaxed marginal constraints. The algorithm comes with theoretical analysis and empirical evidence. 
\end{abstract}
\section{Introduction}

Optimal transport problem was initially formalized by the mathematician Gaspard Monge \cite{monge1781memoire}. Later a series of significant contribution in transportation theory leads to deep connections with more mathematical branches including partial differential equations, geometry, probability theory and statistics \cite{kantorovich1942translation,brenier1991polar}. Optimal transport provides a flexible framework for comparing probability measures. Monge and Kantorovich formulate the optimal transport problem in different ways, in which the Kantovorich formulation is a generalisation of Monge. For the Kantovorich's optimal transport problem, given two probability measures $\mu$ and $\nu$ defined on spaces $X$ and $Y$ respectively, and a cost function $c(x,y): X \times Y \to \mathbb{R}$, which measures the expense of moving one unit of mass from $x \in X$ to $y\in Y$, we aim at finding a joint distribution $\gamma^*$ defined on $Z = X \times Y$ such that the expectation of the cost over the joint distribution is minimized: 
\begin{equation}\label{eq:ot}
  \inf_{\gamma} \cbr[2]{\int_{X\times Y}c(x,y) \,{\mathrm{d}}\gamma (x,y) \Big\vert ~ \gamma \in \Pi (\mu ,\nu ) } .
\end{equation}
The marginal constraints are given by
\begin{equation}
  \label{eq:Pi}
  \Pi(\mu, \nu) = \cbr[2]{\gamma\, \Big\vert \, \gamma\in\mathcal{P}(X\times Y ), \,  \int_{Y} \gamma \dif y = \mu,\, \int_{X} \gamma \dif x = \nu}.
\end{equation}
 Here $\mathcal{P}(X\times Y )$ denotes the set of probabilities defined on $X \times  Y$.  In this work, we only consider $X = Y = \mathbb{R}^d$. The optimal value of the objective in \eqref{eq:ot} is defined as the Wasserstein distance between probability measures $\mu$ and $\nu$. 

In recent years, researchers in applied science fields also discover the importance of optimal transport. In spite of elegant theoretical results, generally computing Wasserstein distance is not an easy task. Computing the discrete optimal transport problem in a straightforward way leads to solving a linear programming problem whose computation cost can be unaffordable with large scale problem settings \cite{pele2009fast}.  In \cite{cuturi2013sinkhorn}, the author smooths the discrete OT problem with an entropic regularization, and designs a fast matrix scaling algorithm which demonstrates high efficiency. However, the computation can be even intractable when it comes to continuous case. In \cite{li2018parallel}, the authors compute the continuous problem by first discretizing the space, such treatment is unrealistic for many applications involving probabilistic measures lying in high dimensional space with complicated shapes. 
We have witnessed the success of deep neural networks in dealing with the large scale continuous OT problem \cite{arjovsky2017wasserstein,seguy2017large}. But is it possible to save some efforts for parameter tuning and deal with the continuous problem from another perspective? 

In this paper, instead of solving the standard continuous transport problem, we start with an entropy transport problem as a relaxed optimal transport problem with soft marginal constraints. Recently, the importance of entropy transport problem has drawn researchers' attention as people figuring out its duality connection with unbalanced optimal transport problem \cite{chizat2018unbalanced,liero2018optimal} and treat it as a canonical distance function on the space of positive measures \cite{liero2018optimal}. With these soft marginal constraints, we can realize the corresponding Wasserstein gradient flow as a time evolution Partial Differential Equation (PDE) and finally numerically solve the regularized problem by evolving an interacting particle system via Kernel Density Estimation techniques \cite{parzen1962estimation}. 
To get samples from optimal coupling, the traditional methods like Linear Programming \cite{oberman2015efficient,schmitzer2016sparse,walsh2017general} or Sinkhorn \cite{cuturi2013sinkhorn} usually start with the discretization of the whole continuous space and compute the transport plan for discrete setting as the approximation of the continuous case. Our algorithm can directly output the sample approximation of the optimal coupling without any discretization or training process as neural network method \cite{seguy2017large,korotin2019wasserstein,makkuva2020optimal}. This is also very different from other traditional methods like Monge-Amp\`{e}re Equation \cite{benamou2014numerical} or dynamical scheme \cite{benamou2000computational,li2018parallel,ruthotto2020machine}. We note that a recent independent work \cite{daaloul2021sampling} on sampling algorithm for Wasserstein Barycenter problems shares similar ideas with our proposed method.

Our main contribution is to analyze the theoretical properties of the entropy transport problem constrained on probability space and construct the corresponding Wasserstein gradient flow. For the constrained transport problem, we prove the existence and uniqueness of the solution under certain assumptions, and further study the $\Gamma$-convergence property of the entropy transport functionals to the classical optimal transport functional. Based on the gradient flow we derive, we propose an innovative algorithm for obtaining the sample approximation for the optimal plan. Our method can deal with optimal transport problem between two known densities. As far as we know, despite the classical discretization methods \cite{benamou2000computational,benamou2014numerical,li2018parallel} there is no scalable way to solve this type of problem. We also demonstrate the efficiency of our method by numerical experiments.

The paper is structured as follows. We introduce the constrained entropy transport as a regularized optimal transport problem in section 2; We carry out the Wasserstein gradient flow approach and its particle formulation in section 3; We design the algorithm as an interacting particle system in section 4; and demonstrate its accuracy with empirical evidences in section 5 before the conclusion. 

\section{Constrained Entropy Transport as a regularized Optimal Transport problem} \label{section introduction to ET}

\subsection{Entropy Transport problem}
In our research, we will mainly restrict our discussion on Euclidean Space (i.e. $\mathbb{R}^d$). We denote $\mathcal{M}(\mathbb{R}^d)$ as the space of finite positive Radon measures on $\mathbb{R}^d $. We denote $\mathcal{P}(\mathbb{R}^d)$ as the space of probability measures defined on $ \mathbb{R}^d$.

For $\mu,\nu\in\mathcal{M}(\mathbb{R}^d)$ and $\gamma\in\mathcal{M}(\mathbb{R}^d\times\mathbb{R}^d)$, 
We denote $\pi_1:\mathbb{R}^d\times\mathbb{R}^d\rightarrow \mathbb{R}^d$ as the projection onto the first coordinate: $\pi_1(x,y)=x$; and $\pi_2$ as the projection onto the second coordinate. $\#$ represents the push-forward operation\footnote{Suppose $T:\mathbb{R}^m\rightarrow\mathbb{R}^n$ is a measurable map, suppose $p$ is a measure defined on $\mathbb{R}^m$. Then $T_{\#}p$ is a measure defined on $\mathbb{R}^n$ satisfying: $T_{\#}p(E)=p(T^{-1}(E))\quad\forall~\textrm{measurable set}~E\subset\mathbb{R}^n$.
}. Let us consider the following functional:
\begin{align}
 \mathcal{E}(\gamma|\mu,\nu) =  \iint_{\mathbb{R}^d\times\mathbb{R}^d}c(x,y)d\gamma(x,y) +D({\pi_1}_{\#}\gamma \|\mu)+D({\pi_2}_{\#}\gamma \|\nu).\label{general ET functional}
\end{align}
Here $c:\mathbb{R}^d\times\mathbb{R}^d\rightarrow [0,+\infty]$ is a lower semicontinuous cost function.
$D(\cdot\|\cdot):\mathcal{M}(\mathbb{R}^d)\times\mathcal{M}(\mathbb{R}^d)\rightarrow \mathbb{R}$ is the divergence functional defined as:
\begin{equation}
  D(\mu\|\nu) = 
  \begin{cases}
  \int_{\mathbb{R}^d} F\left(\frac{d\mu}{d\nu}\right)d\nu \quad \textrm{if}~\mu\ll\nu\\
  +\infty \quad \textrm{otherwise}
  \end{cases}  \label{general divergence functional}
\end{equation}
Here $F:[0,+\infty)\rightarrow [0,+\infty]$ is some convex function and there exists at least one $s>0$ such that $F(s)<+\infty$. 

Then the general \textbf{ Entropy Transport problem} can be formulated as:
\begin{equation}
 \inf\limits_{\gamma\in\mathcal{M}(\mathbb{R}^d\times\mathbb{R}^d)}\left\{ \mathcal{E}(\gamma|\mu,\nu) \right\}. \label{general ET}
\end{equation}
It is not hard to show that $\mathcal{E}$ is convex on $\mathcal{M}(\mathbb{R}^d\times\mathbb{R}^d)$:
\begin{theorem}\label{thm convexity of ET functional}
Under the previous assumptions on $c$ and $F$, for any$\gamma_a,\gamma_b\in\mathcal{M}(\mathbb{R}^d\times\mathbb{R}^d)$ and $0\leq t\leq 1$, we have:
\begin{equation*}
\mathcal{E}(t\gamma_a+(1-t)\gamma_b)\leq t\mathcal{E}(\gamma_a)+(1-t)\mathcal{E}(\gamma_b).
\end{equation*}
\end{theorem}
We have the following theorem on existence and uniqueness of minimizer for problem \eqref{general ET}:
\begin{theorem}\label{thm existence uniqueness sol }
  We consider problem \eqref{general ET} involving the entropy transport functional defined in \eqref{general ET functional}. Suppose that the cost $c$ and $F$ satisfy the previous assumptions. We further assume that there exists at least one $\gamma\in\mathcal{M}(\mathbb{R}^d\times\mathbb{R}^d)$ such that $\mathcal{E}(\gamma|\mu,\nu)<+\infty$. Then the problem \eqref{general ET} admits at least one optimal solution.\\
  If we further assume $c(x,y)=h(x-y)$ with strictly convex $h:\mathbb{R}^d\rightarrow[0,+\infty)$; $F$ is strictly convex, and is superlinear, i.e. $\lim_{s\rightarrow+\infty}\frac{F(s)}{s}=+\infty$; distribution $\mu$,$\nu$ has density functions, i.e. $\mu\ll\mathscr{L}^d,\nu\ll\mathscr{L}^d$ where $\mathscr{L}^d$ denotes the Lebesgue measure on $\mathbb{R}^d$. Under these further assumptions, there exists unique optimal solution to the problem \eqref{general ET}.
\end{theorem}

There could be many choices for cost $c$ and divergence $D$. For example, setting $c(x,y)=-\log\cos^2(\min    \{\frac{|x-y|}{2\tau},\pi  \})$ and $F(s)=s\log s-s+1$ leads to an entropy transport problem equivalent to solving for the Wasserstein-Fisher-Rao (or Hellinger–Kantorovich) distance between distributions $\mu$ and $\nu$ \cite{peyre2019computational}\cite{chizat2018unbalanced}\cite{liero2018optimal}.

In our research, we mainly treat the entropy transport problem \eqref{general ET} as a relaxed optimal transport problem with soft marginal constraints. Recall that optimal transport problem is formulated as:
\begin{equation}
  \inf_{\substack{\text{$\gamma\in\mathcal{P}(\mathbb{R}^d\times\mathbb{R}^d)$, } \\ \text{ $\pi_{1 \#}\gamma=\mu,\pi_{2 \#}\gamma=\nu$} }}\iint c(x,y)~d\gamma(x,y) \label{OT}.
\end{equation}
Here $\mu,\nu\in\mathcal{P}(\mathbb{R}^d)$. \eqref{OT} can also be treated as an entropy transport problem:
\begin{align}
  &\inf_{\gamma\in\mathcal{M}(\mathbb{R}^d\times\mathbb{R}^d)}\left\{\mathcal{E}_{\iota}(\gamma|\mu,\nu)\right\} \label{iota ET}  \\ \textrm{with}\quad 
  \mathcal{E}_{\iota}(\gamma|\mu,\nu) &= \iint c(x,y)d\gamma(x,y)  
  +\int\iota\left(\frac{d\pi_{1 \#}\gamma}{d\mu}\right)d\mu + \int\iota\left(\frac{d\pi_{1 \#}\gamma}{d\nu}\right)d\nu\label{iota ET functional}
\end{align}
Here we choose $F(\cdot)=\iota(\cdot)$ in the original functional \eqref{general ET functional} with $\iota$ defined as:
\begin{equation*}
  \iota(s) = \begin{cases}
    0 \quad \textrm{if}~ s\neq 1\\
    +\infty \quad \textrm{if}~ s=1
  \end{cases} 
\end{equation*}
It is not hard to verify that \eqref{iota ET} and \eqref{OT} are equivalent. 

In order to derive a relaxed optimal transport problem as an entropy transport problem, we relax the divergence terms $D(\pi_{1 \#}\gamma\|\mu),D(\pi_{2 \#}\gamma \| \nu)$ involving marginal distributions of $\gamma$: For example, we may replace the function $\iota(\cdot)$ with $\Lambda F(\cdot)$, here $\Lambda>0$ is a large positive number and $F$ is a smooth convex function with $F(1)=0$ and 1 is the unique minimizer. There are definitely many choices for $F$, some popular choices are the power-like entropies \cite{liero2018optimal}:
\begin{equation*}
  F_\alpha(s) = \begin{cases}
   \frac{1}{\alpha(\alpha-1)}(s^\alpha-\alpha(s-1)-1) ~ \textrm{if}~ \alpha>1 \\
   s\log s -s +1 ~\textrm{if}~ \alpha = 1
  \end{cases}
\end{equation*}
In our research, we mainly focus on the case when $\alpha = 1$ since it is a canonical divergence functional in transportation theory and enables us to establish corresponding theoretical results. On the other hand, $\alpha = 1$ leads to more concise form when we are deriving for our algorithms and can reduce the computational cost. It also worth mentioning that when $\alpha=1$, the corresponding divergence $D(\cdot \| \cdot )$ is usually called Kullback-Leibler (KL) divergence \cite{kullback1951information} and we will denote it as $D_{\rm{KL}}(\cdot\|\cdot)$.

From now on, we should focus on the following functional involving cost $c$ and enforcing the marginal constraints by using KL-divergence:
\begin{equation}
  \mathcal{E}_{\Lambda,\rm{KL}}(\gamma|\mu,\nu) = \iint_{\mathbb{R}^d\times\mathbb{R}^d} c(x,y)~d\gamma(x,y) + \Lambda D_{\rm{KL}}(\pi_{1 \#} \gamma \| \mu ) + \Lambda D_{\rm{KL}}(\pi_{2 \#} \gamma \| \nu ). \label{KL ET functional}
\end{equation}
Since a majority of our applications are devoted to optimal transport problems with specific form of cost functions $c$ and marginals $\mu,\nu$, let us make the following assumptions in order to make our future discussion concise:
\begin{align}
   (A).~& c(x,y) = h(x-y)~\textrm{with}~ h ~\textrm{a strictly convex function}.\nonumber\\
   \quad (B). ~& \mu,\nu\in\mathcal{P}(\mathbb{R}^d)~\rm{and}~   \mu\ll\mathscr{L}^d, \nu\ll\mathscr{L}^d ~ \label{additional cond}
\end{align}
Here $\mathscr{L}^d$ is Lebesgue measure on $\mathbb{R}^d$.
\begin{theorem}\label{uniqueness optimal distribution OT}
 Under additional assumptions \eqref{additional cond}, the optimal transport problem \eqref{OT} admits a unique solution.
\end{theorem}
The proof of this theorem can be found in \cite{ambrosio2008gradient} Theorem 6.2.4.

\subsection{Constrained Entropy Transport problem and some of its properties} \label{subsec: PET}
We now restrict the functional $\mathcal{E}_{\Lambda,\rm{KL}}(\cdot|\mu,\nu)$ to the probability manifold $\mathcal{P}(\mathbb{R}^d\times\mathbb{R}^d)$ instead of $\mathcal{M}(\mathbb{R}^d\times\mathbb{R}^d)$. There are mainly two reasons of such restriction:
\begin{itemize}
  \item This will allow us to compute the Wasserstein gradeint flow of $\mathcal{E}_{\Lambda,\rm{KL}}(\cdot|\mu,\nu)$ on probability manifold and will later lead to our algorithm in form of an interacting particle system (See section \ref{flow of ET}); 
  \item When restricting on $\mathcal{P}(\mathbb{R}^d\times\mathbb{R}^d)$, there is a natural $\Gamma-$convergence of $\{\mathcal{E}_{\Lambda,\rm{KL}}( \cdot | \mu , \nu)\}_{\Lambda}$ to the entropy transport functional corresponding to Optimal Transport problem \eqref{OT} as $\Lambda\rightarrow +\infty$ (See Theorem \ref{thmGamma});
\end{itemize}
We thus consider the following optimization problem:
\begin{equation} 
  \inf_{\gamma\in\mathcal{P}(\mathbb{R}^d\times\mathbb{R}^d)}\left\{ \mathcal{E}_{\Lambda,\rm{KL}}(\gamma |\mu , \nu)  \right\}
  \label{constrained-ET}.
\end{equation}
In the following discussion, we will call such problem \eqref{constrained-ET} \textbf{constrained Entropy Transport problem} since we constrain the space for $\gamma$ on $\mathcal{P}(\mathbb{R}^d\times\mathbb{R}^d)$.

We now denote $\mathcal{E}_{\min}=\inf_{\gamma\in\mathcal{P}(\mathbb{R}^d\times\mathbb{R}^d)}\{\mathcal{E}_{\Lambda,\rm{KL}}(\gamma| \mu ,   \nu)\}$. It can be shown that this infimum value $\mathcal{E}_{\min}$ is finite, i.e. $\mathcal{E}_{\min}>-\infty$. The following theorem shows the existence of the optimal solution to problem \eqref{constrained-ET}. It also describes the relationship between the solution 
to the constrained Entropy Transport problem \eqref{constrained-ET} and the solution to the general Entropy Transport problem \eqref{general ET}:
\begin{theorem}\label{main thm A}
  Suppose  $\tilde{\gamma}$ is the solution to original entropy transport problem \eqref{general ET}. Then we have $\tilde{\gamma}=Z\gamma$, here $Z = e^{-\frac{\mathcal{E}_{\rm{min}}}{2\Lambda}       } $ and $\gamma\in\mathcal{P}(\mathbb{R}^d\times\mathbb{R}^d)$ is the solution to constrained Entropy Transport problem \eqref{constrained-ET}.
\end{theorem}
The following corollary 
guarantees the uniqueness of optimal solution to \eqref{constrained-ET}:
\begin{corollary}\label{corollary uniqueness constrained - ET}
  The constrained Entropy Transport problem admits a unique optimal solution.
\end{corollary}
The proof of Theorem \ref{main thm A} and Corollary \ref{corollary uniqueness constrained - ET} are provided in Appendix \ref{App a2}.

We can now characterize the structure of the optimal solution to problem \eqref{constrained-ET}:
\begin{theorem}[Characterization of optimal distribution $\gamma_{cET}$ to problem \eqref{constrained-ET} ]\label{main part characterization of optimal gamma}
  We assume $\textrm{supp}(\mu)=\textrm{supp}(\nu)=\mathbb{R}^d$. Suppose $\gamma_{cET}$ solves the constrained Entropy Transport problem \eqref{constrained-ET}. Then there exist certain $\varphi,\psi \in B(\mathbb{R}^d;\mathbb{R})  $ satisfying: 
  $\varphi(x)+\psi(y)\leq c(x,y)$ for any $x,y\in\mathbb{R}^d$, such that:
  \begin{align}
    \varphi(x)+\psi(y)=c(x,y) & \quad  \gamma_{cET}  - \textrm{almost surely}, \nonumber \\
    \pi_{1 \#}\gamma_{cET} =e^{\frac{   \mathcal{E}_{   \textrm{min } }  - 2\varphi }{2\Lambda}} \mu & \quad   \pi_{2 \#}\gamma_{cET}  = e^{\frac{   \mathcal{E}_{   \textrm{min } }  - 2\psi }{2\Lambda}}\nu \label{main part:marginal entropy transport on P}
  \end{align}
\end{theorem}
We provide a direct proof of this theorem in Appendix \ref{App a2}.
We can compare the structure of solution $\gamma_{cET}$ to \eqref{constrained-ET} with the solution $\gamma_{OT}$
to the Optimal Transport problem \eqref{OT}.
\begin{theorem}[Characterization of optimal distribution $\gamma_{OT}$ to problem \eqref{OT}]\label{main part characterization of optimal distribution to Optimal Transport}
If we assume additional condition on the cost function: $c(x,y)\leq a(x)+b(y)$ with $a\in L^1(\mu)$, $b\in L^1(\nu)$. Then there exists an optimal distribution $\gamma_{OT}$ to problem \eqref{OT}. There exist $\varphi,\psi\in C(\mathbb{R}^d)$ such that $\varphi(x)+\psi(y)\leq c(x,y)$ for any $x,y\in\mathbb{R}^d$ with:
\begin{align}
  \varphi(x)+\psi(y)=c(x,y)&\quad \gamma_{OT}-\textrm{almost surely}\nonumber  \\
  \pi_{1 \#}\gamma_{OT} = \mu& \quad \pi_{2 \#}\gamma_{OT} = \nu \label{main part:marginal optimal transport on P}
\end{align}
\end{theorem}
Since we are using constrained Entropy Transport problem \eqref{constrained-ET} to approximate Optimal Transport problem \eqref{OT}, we are interested in comparing the difference between their optimal distributions $\gamma_{c ET}$ and $\gamma_{OT}$. Although we can identify their difference from marginal conditions \eqref{main part:marginal entropy transport on P} and \eqref{main part:marginal optimal transport on P}, we currently do not have a quantitative analysis on the difference between $\gamma_{cET}$ and $\gamma_{OT}$. This may serve as one of our future research directions.

Despite the discussions on  the constrained problem \eqref{constrained-ET} with fixed $\Lambda$, we also establish asymptotic results for \eqref{constrained-ET} with quadratic cost $c(x,y)=|x-y|^2$ as $\Lambda \rightarrow +\infty$. For the rest of this section, we define:
\begin{equation*}
   \mathcal{P}_2(X) = \left\{\gamma \Big| ~\gamma\in\mathcal{P}(X),\gamma\ll\mathscr{L}^{ m},~\int_{X} |x|^2d\gamma<+\infty \right\}
\end{equation*}
where $X=\mathbb{R}^{m}$ is any $m$ dimensional Euclidean space.
Let us now consider $\mathcal{P}_2(\mathbb{R}^d\times\mathbb{R}^d)$ and assume it is equipped with the topology of weak convergence. We are able to establish the following $\Gamma$-convergence results for the functional $\mathcal{E}_{\Lambda,\rm{KL}}( \cdot|\mu,\nu)$ defined on $\mathcal{P}_2(\mathbb{R}^d\times\mathbb{R}^d)$:
\begin{theorem}[$\Gamma$-convergence]\label{thmGamma}
Suppose $c(x,y)=|x-y|^2$. Assume that we are given $\mu,\nu\in\mathcal{P}_2(\mathbb{R}^d)$ and at least one of $\mu$ and $\nu$ satisfies the Logarithmic Sobolev inequality with constant $ K >0$. Let  $\{\Lambda_n\}$ be a positive increasing sequence, satisfying $\lim_{n\rightarrow \infty}\Lambda_n = +\infty$. We consider the sequence of functionals $\{ \mathcal{E}_{\Lambda_n,\rm{KL}}(\cdot|\mu,\nu) \}$. Recall the functional $\mathcal{E}_{\iota}(\cdot|\mu,\nu)$ defined in \eqref{iota ET functional}. Then $\{ \mathcal{E}_{\Lambda_n,\rm{KL}}(\cdot|\mu,\nu)\}$ $\Gamma$- converges to $\mathcal{E}_{\iota}(\cdot|\mu,\nu)$ on $\mathcal{P}_2(\mathbb{R}^d\times\mathbb{R}^d)$.
\end{theorem}
We can further establish the equi-coercive property for the family of functionals $\{\mathcal{E}_{\Lambda_n,\rm{KL}}(\cdot|\mu,\nu)\}_n$ and we use the Fundamental Theorem of $\Gamma$-convergence \cite{dal2012introduction} \cite{braides2006handbook} to establish the following asymptotic results:
\begin{theorem}[Property of $\Gamma$-convergence]\label{main thm B}
Suppose $c(x,y)=|x-y|^2$. Assuming $\mu,\nu\in\mathcal{P}_2(\mathbb{R}^d)$ and both $\mu, \nu$ satisfy the Logarithmic Sobolev inequality with constants $ K_\mu,K_\nu>0$. According to Corollary \ref{thm existence uniqueness sol }, problem \eqref{constrained-ET} with functional $\mathcal{E}_{\Lambda_n,\rm{KL}}(\cdot | \mu,\nu   )$
admits a unique optimal solution, let us denote it as $\gamma_n$. According to Theorem \ref{uniqueness optimal distribution OT}, the Optimal Transport problem \eqref{OT} also admits a unique solution, we denote it as $\gamma_{OT}$. Then: $\lim_{n\rightarrow \infty}\gamma_n = \gamma_{OT}$ in $\mathcal{P}_2(\mathbb{R}^d\times \mathbb{R}^d)$.
\end{theorem}
Detailed proofs of these theorems regarding $\Gamma$-convergence are provided in Appendix\ref{App a3}.

\section{Wasserstein Gradient Flow Approach for Solving the Regularized Problem} \label{flow of ET}
\subsection{Wasserstein gradient flow}\label{Wass mfld and Wass grad flow}
 There are already numerous researches \cite{jordan1998variational,otto2001,ambrosio2008gradient} regarding Wasserstein gradient flows of different types of functionals defined on the Wasserstein manifold-like structure $(\mathcal{P}_2(\mathbb{R}^d),g^W)$ that successfully relate certain kinds of time evolution Partial Differential Equations (PDEs) to the manifold gradient of corresponding functionals on $(\mathcal{P}_2(\mathbb{R}^d), g^W   )$. The Wasserstein manifold-like structure is the manifold $\mathcal{P}_2(\mathbb{R}^d)$ equipped with a special metric $g^W$ induced by the 2-Wasserstein distance. Under this setting, the Wasserstein gradient flow of a certain functional $\mathcal{F}$ defined on $\mathcal{P}_2(\mathbb{R}^d)$ can thus be formulated as:
 \begin{equation}
    \frac{\partial\gamma_t}{\partial t} = -\textrm{grad}_W \mathcal{F}(\gamma_t) 
    \label{Wasserstein Gradient Flow}
 \end{equation}
 One can explain this equation \eqref{Wasserstein Gradient Flow} as the continuous steepest descent algorithm applied to $\mathcal{F}$ in order to determine the minimizer of the target functional $\mathcal{F}$. For more detailed information regarding Wasserstein manifold-like structure and Wasserstein gradients, please check Appendix \ref{App b1}.

 \subsection{Wasserstein gradient flow of Entropy Transport functional}\label{sec3.2}
 We now come back to our constrained entropy transport problem \eqref{constrained-ET}. There are mainly two reasons why we choose to compute the Wasserstein gradient flow of functional $\mathcal{E}_{\Lambda,\rm{KL}}( \cdot |\mu,\nu )$:
 \begin{itemize}
   \item Computing the Wasserstein gradient flow is equivalent to applying gradient descent to determine the minimizer of the entropy transport functional \eqref{KL ET functional};
   \item In most of the cases, Wasserstein gradient flows can be realized as a time evolution PDE describing the density evolution of a stochastic process. As a result, once we derived the gradient flow, there will be a natural particle version associated to the gradient flow. And this will make the computation of gradient flow tractable since we can evolve the particle system by applying the forward Euler scheme.
 \end{itemize}
 Now let us compute the Wasserstein gradient flow of $\mathcal{E}_{\Lambda,\rm{KL}}(\cdot|\mu,\nu)$:
 \begin{equation}
    \frac{\partial\gamma_t }{\partial t}=-\textrm{grad}_W\mathcal{E}_{\Lambda,\rm{KL}}(\gamma_t|\mu,\nu) , ~~ ~   \gamma_t|_{t=0}=\gamma_0 \label{wass grad flow ET functional }
 \end{equation}
 To keep our notations concise, we denote $\rho(\cdot,\cdot,t) = \frac{d\gamma_t}{d\mathscr{L}^{2d}}$ $\varrho_1 = \frac{d\mu}{d\mathscr{L}^d}$, $\varrho_2 = \frac{d\nu}{d\mathscr{L}^d}$, we can show that the previous equation \eqref{wass grad flow ET functional } can be written as:
  \begin{equation}
   \frac{\partial \rho}{\partial t}=\nabla\cdot(\rho~\nabla (c(x,y)+ \Lambda \log(\frac{\rho_1(x,t)}{\varrho_1(x)}) + \Lambda \log(\frac{\rho_2(y,t)}{\varrho_2(y)}) )) 
  \label{PDE vers}
 \end{equation}
 Here $\rho_1(\cdot,t) = \frac{d\pi_{1 \#}\gamma_t}{d\mathscr{L}^d} =\int\rho(\cdot,y,t)dy$ and $\rho_2(\cdot,t) = \frac{d\pi_{2   \#}\gamma_t}{d\mathscr{L}^d}=\int\rho(x,\cdot,t)dx$ are density functions of marginals of $\gamma_t$. We put the details of our derivation in 
 Appendix \ref{App b2}.
 \begin{remark}\label{no displacement convexity}
   We are currently not clear about the displacement convexity of the functional $\mathcal{E}_{\Lambda, \rm{KL}  }(\cdot | \mu,\nu   )$ on Wasserstein manifold-like structure $(\mathcal{P}_2(\mathbb{R}^d\times\mathbb{R}^d),g^W)$, which will guarantee its gradient flow to converge at its minimizer. This will be one of our future research directions . In practice, we should rely on the computational results to tell us whether our method works properly.
 \end{remark}
 
 \subsection{Particle formulation of our derived gradient flow} \label{sectionSDE}
 Let us treat \eqref{PDE vers} as certain kind of continuity equation, i.e. we treat $\rho(\cdot,t)$ as the density of the time-evolving random particles. Then the vector field that drives the random particles at time $t$ should be $-\nabla(c(x,y)+\Lambda\log\left(\frac{\rho_1(x,t)}{\varrho_1(x)}\right)+\Lambda\log\left(\frac{\rho_2(y,t)}{\varrho_2(y)}\right))$. This helps us design the following dynamics $\{(X_t,Y_t)\}_{t\geq 0}$: (here $\dot X_t$ denotes the time derivative $\frac{d X_t}{dt}$)
 \begin{equation}
 \label{SDE version}
    \begin{cases}
    \dot X_t= -\nabla_x c(X_t,Y_t) + \Lambda (\nabla\log\varrho_1(X_t) -\nabla\log \rho_1(X_t,t)) ;\\
    \dot Y_t= -\nabla_y c(X_t,Y_t) + \Lambda (\nabla\log\varrho_2(Y_t) -\nabla\log \rho_2(Y_t,t)) ;
    \end{cases} 
 \end{equation}
 where $\textrm{Law}(X_0,Y_0) =  \gamma_0 $. Here $\rho_1(\cdot,t)$ is the density of $\textrm{Law}(X_t)$ and $\rho_2(\cdot,t)$ is the density of $\textrm{Law}(Y_t)$.
 If we assume the process \eqref{SDE version} is well-defined, then the probability density $\rho_t(x,y)$ of $(X_t,Y_t)$ should solve the PDE \eqref{PDE vers}.

 When we take a closer look at \eqref{SDE version}, we can verify that the movement of particle $(X_t,Y_t)$ at certain time $t$ depends on the probability density of $\textrm{Law}((X_t,Y_t))$ at $(X_t,Y_t)$, which can be approximated by the distribution of the surrounding particles near $(X_t, Y_t)$.
 Such equation \eqref{PDE vers} can be treated as 
 a limit case of aggregation-diffusion equation \cite{carrillo2019aggregation,carrillo2019blob} with Dirac kernel convolution.
Generally speaking, we plan to evolve \eqref{SDE version} as a particle aggregation model in order to converge to a sample-wised approximation of the Optimal Transport plan $\gamma_{OT}$ for OT problem \eqref{OT}.
 
 We expect that the distribution $\textrm{Law}((X_t,Y_t))$ converges to $\gamma_{ cET} $ as $t\rightarrow\infty$, here $\gamma_{ cET }  $ is the minimizer of functional $\mathcal{E}_{\Lambda,\rm{KL}}(\cdot|\mu,\nu)$ on $\mathcal{P}(\mathbb{R}^d\times\mathbb{R}^d)$. Recall the discussion in section \ref{subsec: PET}, $\gamma_{cET}   $ can be treated as an approximation to the solution of OT problem \eqref{OT}.
 
 In conclusion, evolving \eqref{SDE version} as an interacting particle system will provide a particle aggregation method for computing the sample-wise approximation to the Optimal Transport problem \eqref{OT}.

\section{Algorithmic Development}
\subsection{Numerical Approximation via Kernel Method}
To use the Euler scheme to simulate the stochastic process \eqref{SDE version}, we have to find a numerical approximation for the term $\nabla \log \rho(x)$. Here we use the Kernel Density Estimation \cite{parzen1962estimation} to approximate the density by convolving with kernel $K(x,\xi)$. In this paper, we simply choose the Radial Basis Function (RBF) kernel:
\begin{equation*}
  K(x,\xi)=\exp\left(-\frac{|x-\xi|^2}{2\tau^2}\right)
\end{equation*} 
Then we approximate $\nabla\log\rho$ by:
\begin{equation}
\begin{split}
  \nabla\log\rho(x)\approx\nabla\log(K*\rho)(x)=\frac{(\nabla_x K)*\rho(x)}{K*\rho(x)} \label{kernel_approximation}
  \end{split}
\end{equation}
Here $K*\rho(x) = \int K(x,\xi)\rho(\xi)d\xi$, $(\nabla_xK)*\rho(x)=\int \nabla_x K(x,\xi)\rho(\xi)d\xi$\footnote{Notice that we always use $\nabla_x K$ to denote the partial derivative of $K$ with respect to the first components.}. Such technique is also known as blobing method, which was first studied in \cite{carrillo2019blob} and has already been applied to Bayesian sampling \cite{chen2018unified}. With this reformulation, we  can evaluate the gradient log density function based on the locations of the particles: 
\begin{equation*}
    \frac{\mathbb{E}_{\xi\sim\rho}\nabla_x K(x,\mathbf{\xi})}{\mathbb{E}_{\xi\sim\rho}K(x,\xi)}\approx\frac{\sum_{k=1}^N\nabla_x K(x,\xi_k)}{\sum_{k=1}^N K(x,\xi_k)} ~~~\xi_1,...,\xi_N, \textrm{i.i.d.}\sim\rho
\end{equation*}

With the help of this method, we are able to construct the following particle system involving $N$ particles $\{(X_i , Y_i)\}_{i=1,...,N}$. For the $i$-th particle, we have:
\begin{equation}
\left\{
\begin{aligned}
\dot X_i(t) &= -\nabla_x c(X_i(t),Y_i(t))   - \Lambda \left(\nabla V_1(X_i(t)) +  \frac{\sum_{k=1}^N\nabla_x K(X_i(t),X_k(t))}{\sum_{k=1}^N K(X_i(t),X_k(t))} \right) \\ 
\dot Y_i(t) &= -\nabla_y c(X_i(t),Y_i(t))   - \Lambda \left(\nabla V_2(Y_i(t)) +  \frac{\sum_{k=1}^N\nabla_x K(Y_i(t),Y_k(t))}{\sum_{k=1}^N K(Y_i(t),Y_k(t))} \right)
\end{aligned}
\right. \label{num_particle_sys}
\end{equation}
Here we denote $V_1=-\log\varrho_1$, $V_2=-\log\varrho_2$. Since we only need the gradients of $V_1,V_2$ , our algorithm can deal with unnormalized continuous probability measures. 
We numerically verify that when $t\rightarrow \infty$, the empirical distribution $\frac{1}{N}\sum_{i=1}^N\delta_{(X_i(t),Y_i(t))}$ will converge to the optimal distribution $\gamma_{ cET}$ that solves \eqref{constrained-ET} with sufficient large $N$ and $\Lambda$, while the rigorous proof is reserved for our future work. 

\subsection{More Computational Efficiency with Random Batch Methods }
Taking closer look at the equation \eqref{num_particle_sys}, we can see the main computational efforts are put into approximating the gradient log density function. In each time step, the computational cost is of the order $\mathcal{O}(n^2)$. Inspired by \cite{jin2020random}, we apply the Random Batch Methods (RBM) here to reduce the computational cost. Assume that we have $n$ particles in the system and we can divide all particles into $m$ batches equally. Then in each iteration, we only consider the particles in the same batch as the particle $X_i$ when we evaluate the $\nabla \log \rho(X_i)$. From simple analysis we know that the computational cost now is reduced to the order $\mathcal{O}(n^2/m)$, which is a significant improvement. With proper choice of batch size, we can still get reasonable approximation but spend much less computational efforts. The algorithm scheme is summarized in the algorithm \ref{alg:alg_par_sys}. 

\subsection{Extention to Wasserstein Barycenter Problem}
Our framework can be extend to multi-marginal problems. Suppose we have $m$ marginal distributions $\mu_1,...,\mu_m$ 
with cost function $c(x_1,...,x_m)$. The general multi-marginal problem \cite{gangbo1998optimal} can be formulated as:
\begin{equation}
 \inf_{\gamma\in\Pi(\mu_1,...,\mu_m)}\left\{\int_{\mathbb{R}^{m\times d}}c(x_1,...,x_m)~d\gamma(x_1,...,x_m)\right\}\label{multi marginal problem}
\end{equation}
Here $\Pi(\mu_1,...,\mu_m)$ denotes the set of $\gamma\in\mathcal{P}(\mathbb{R}^{m\times d})$ with its $m$ marginals equal to $\mu_1,...\mu_m$. To deal with \eqref{multi marginal problem}, we extend the entropy transport functional \eqref{KL ET functional} to:
\begin{align}
  \mathcal{E}(\gamma|\mu_1,...,\mu_m) \label{multi marginal functional} 
  = \int_{\mathbb{R}^{m d}} c(x_1,...,x_m)
  d\gamma(x_1...x_m) + \sum_{j=1}^m \Lambda_j D_{\rm{KL}}(\gamma_j\|\mu_j)
\end{align}
Where $\gamma_j = \pi_{j\#}\gamma$. It is natural to extend the constrained Entropy Transport problem \eqref{constrained-ET} to the problem: 
\begin{equation*}
  \min_{\gamma\in\mathcal{P}(\mathbb{R}^{m \times d})} \mathcal{E}(\gamma|\mu_1,...,\mu_m)
\end{equation*}
Similar to the two marginals case, we can derive the Wasserstein flow of functional \eqref{multi marginal functional} on $\mathcal{P}_2(\mathbb{R}^{m\times d})$ and compute its  corresponding particle flow in order to evaluate an approximation to the optimal solution of \eqref{multi marginal problem}.

We now consider applying our particle flow algorithm to Wasserstein barycenter problem \cite{agueh2011barycenters}\cite{cuturi2014fast}, which can be treated as a specific multi-marginal problem. This  barycenter problem is formulated as:
\begin{equation}
\min_{\mu\in\mathcal{P}_2(\mathbb{R}^d)} \sum_{i=1}^m \lambda_i W_2^2(\mu,\mu_i)\label{barycenter original}
\end{equation}
Here $\lambda_i>0$ are the weights. The barycenter problem has an equivalent multi-marginal formulation.  
We consider the following multi-marginal problem:
\begin{equation}
  \min_{ \gamma\in\Pi(\mu_1,...,\mu_m)
  } \left\{\int_{\mathbb{R}^{(m+1)\times d}} \sum_{i=1}^m \lambda_i|x-x_i|^2~d\gamma(x,x_1,...,x_m) \right\} \label{barycenter multimargin}
\end{equation}
Notice that there is no marginal constraint for $\gamma$ on the 0-th "$x$" component. \eqref{barycenter original} and \eqref{barycenter multimargin} are equivalent\cite{agueh2011barycenters} in the following sense : if  $\bar{\mu}$ and $\bar{\gamma}$ are the optimal solutions to \eqref{barycenter original} and \eqref{barycenter multimargin}, then $\pi_{0\#}\bar{\gamma}=\bar{\mu}$.

We now apply our scheme to solve \eqref{barycenter multimargin}. We need to consider the functional:
\begin{align*}
  \mathcal{E}(\gamma|\mu_1,..,\mu_m)=\int_{\mathbb{R}^{(m+1)\times d}} \sum_{j=1}^{m} \lambda_j|x-x_j|^2 ~ d\gamma(x,x_1,...,x_m) +\sum_{j=1}^m \Lambda_j D_{\rm{KL}}(\gamma_j\|\mu_j)
\end{align*}
where $\gamma_j = \pi_{j \#}\gamma$, $j=1,...,m$. The particle system of the Wasserstein gradient flow of this functional can be written as:
\begin{align*}
\begin{cases} 
\dot X^{(0)}_i(t) &= \sum_{j=1}^{m} -2\lambda_j(X_i^{(0)}(t)-X_i^{(j)}(t))\\
\dot X^{(1)}_i(t) &= -\nabla_{x_1} d(X^{(0)}_i(t), X^{(1)}_i(t)) 
\quad + \Lambda_1 \left( \nabla_x \log\varrho_1(X^{(1)}_i(t)) \right. \\ 
&\quad\left.+ \frac{\sum_{k=1}^N\nabla_\xi K(X^{(1)}_i(t),X^{(1)}_k(t))}{\sum_{k=1}^N K(X^{(1)}_i(t),X^{(1)}_k(t))}\right) \\ 
&...\\
\dot X^{(m)}_i(t) &= -\nabla_{x_m} d(X^{(0)}_i(t),X^{(m)}_i(t))
\quad + \Lambda_m \left(\nabla_x \log\varrho_m(X^{(m)}_i(t))\right.\\
&\quad\left.+\frac{\sum_{k=1}^N\nabla_\xi K(X^{(m)}_i(t),X^{(m)}_k(t))}{\sum_{k=1}^N K(X^{(m)}_i(t),X^{(m)}_k(t))}\right)
\end{cases}  
\label{bary center dynamical_system_marginal_density}
\end{align*}
here $i$ goes from $1$ to $N$ . There are $N$ particles $\{(X_i^{(0)},X_i^{(1)},...,X_i^{(m)})\}_{i=1,...,N}$ evolving in $\mathbb{R}^{(m+1)\times d}$ together. As $t\rightarrow \infty$, the empirical distribution of particles $\{X^{(0)}_i\}_{i=1,...,N}$ are expected to be an approximation of barycenter $\bar{\mu}$ of distributions $\mu_1,...,\mu_m$.

\begin{algorithm}[tb]
\caption{Random Batch Particle Evolution Algorithm}
\label{alg:alg_par_sys}
\begin{algorithmic}
   \STATE {\bfseries Input:} The density functions of the marginals $\varrho_1, \varrho_2$, timestep $\Delta t$, total number of iterations $T$, parameters of the chosen kernel $K$
   \STATE {\bfseries Initialize:} The initial locations of all particles $X_i(0)$ and $Y_i(0)$ where $i=1,2,\cdots,n$, 
   \FOR{t\,=\,1,2,$\cdots$,T}
   \STATE Shuffle the particles and divide them into $m$ batches: $\mathcal{C}_1,\cdots,\mathcal{C}_m$
    \FOR{each batch $\mathcal{C}_q$}
    \STATE Update the location of each particle $(X_i,Y_i)$ ($i\in \mathcal{C}_q$)
    \begin{align*}
    \boldsymbol{X}_i &\leftarrow\boldsymbol{X}_i - \Delta t[~ \nabla_x c(X_i(t),Y_i(t)) + \Lambda \nabla V_1(X_i(t))
              +\Lambda \frac{\sum_{k\in \mathcal{C}_q}\nabla_x K(X_i(t),X_k(t))}{\sum_{k\in \mathcal{C}_q} K(X_i(t),X_k(t))} ~] \\ 
    \boldsymbol{Y}_i &\leftarrow\boldsymbol{Y}_i - \Delta t[~ \nabla_y c(X_i(t),Y_i(t)) + \Lambda \nabla V_2(Y_i(t)) 
              +\Lambda \frac{\sum_{k\in \mathcal{C}_q}\nabla_x K(Y_i(t),Y_k(t))}{\sum_{k\in \mathcal{C}_q} K(Y_i(t),Y_k(t))} ~]
    \end{align*}
    \ENDFOR
   \ENDFOR
   \STATE {\bfseries Output:} A sample approximation of the optimal coupling: $X_i(T), Y_i(T)$ for $i=1,2,\cdots,n$
\end{algorithmic}
\end{algorithm}

\section{Numerical Experiments}
In this section, we test our algorithm on several toy examples. All experiments are conducted on a machine with 2.20GHz CPU, 16GB of memory. 

{\bf 1D Gaussian} We set two 1D Gaussian distributions $\mathcal{N}(x;-4,1), \mathcal{N}(x;6,1)$ as marginals and run the algorithm to compute the sample approximation of the optimal transport plan between them. We set $\lambda=200, \Delta t=0.001$ and run it with 1000 particles $(X_i,Y_i)$'s for 1000 iterations. We initialize the particles by drawing 1000 i.i.d. sample points from $\mathcal{N}(x;-20,4)$ as $X_i$'s and 1000 i.i.d. sample points from $\mathcal{N}(x;20,2)$ as $Y_i$'s. The empirical results are shown in figure \ref{fig:1dGauss_marg} and figure \ref{fig:1dGauss_map}. We can see that after 1000 iterations, we get a good sample approximation of the optimal transport plan. 

\begin{figure} 
\centering
{\scriptsize 
\begin{tabular}{ccc}
\includegraphics[width=.28\textwidth]{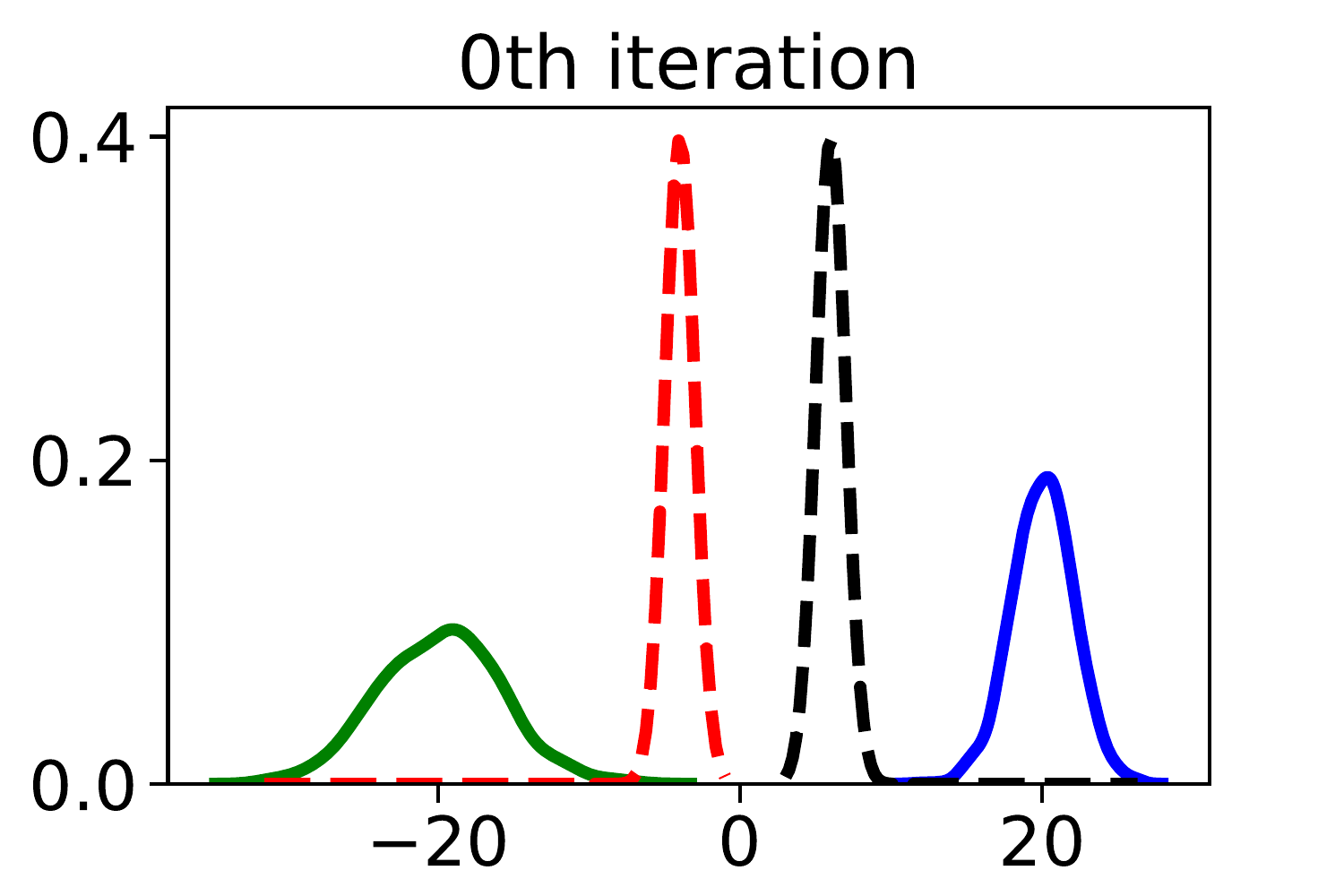} &
\includegraphics[width=.28\textwidth]{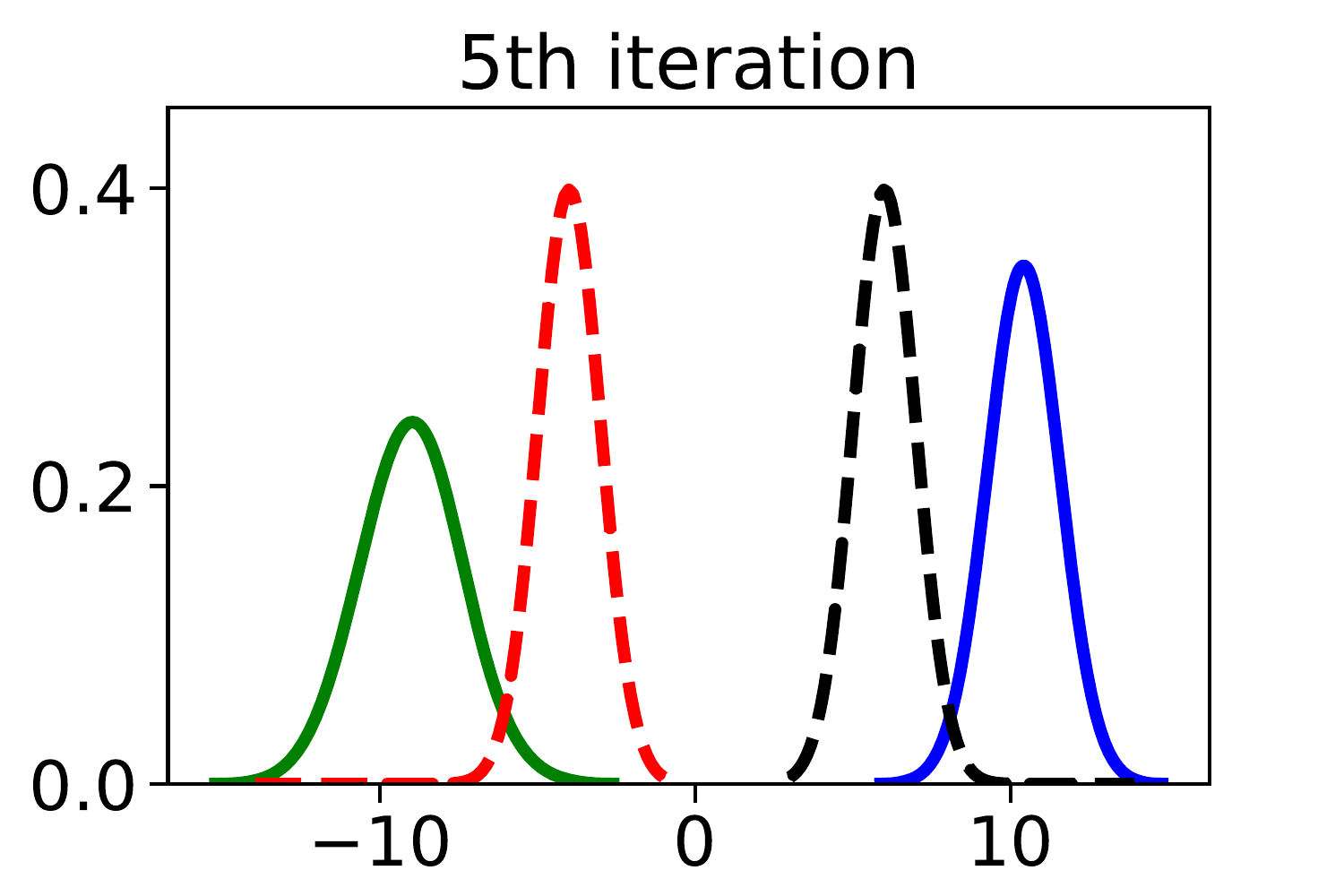} &
\includegraphics[width=.28\textwidth]{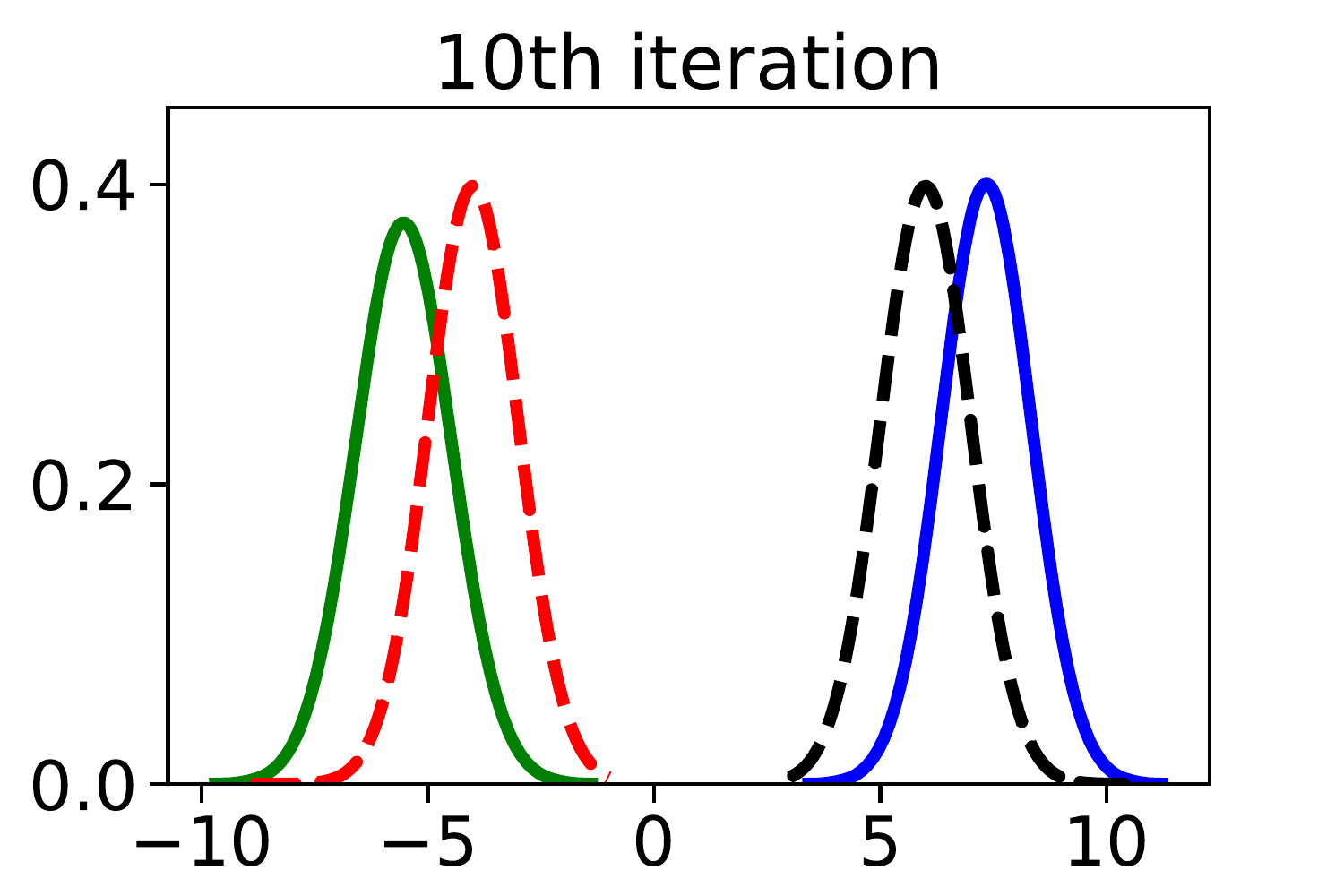}\\
\includegraphics[width=.28\textwidth]{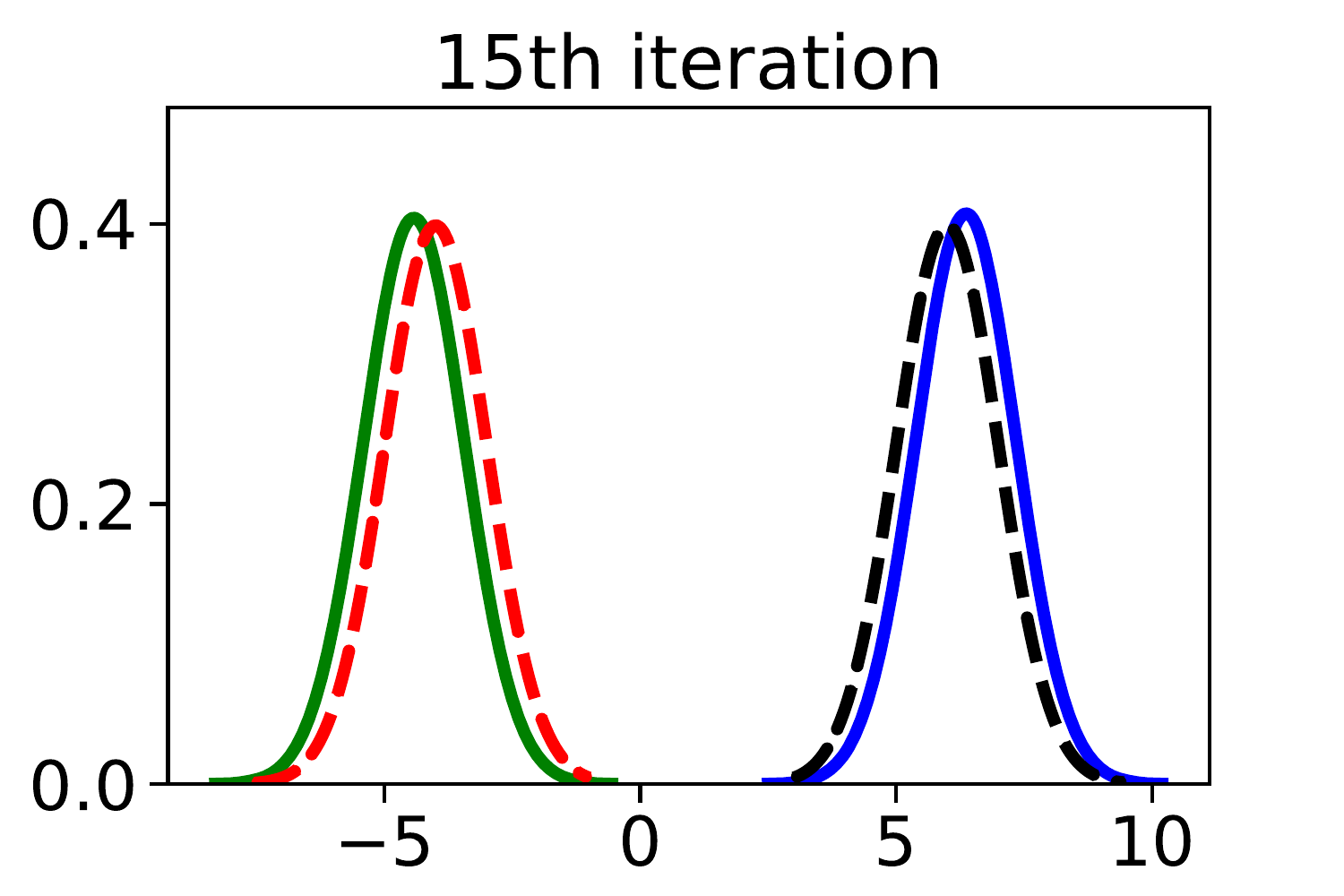} &
\includegraphics[width=.28\textwidth]{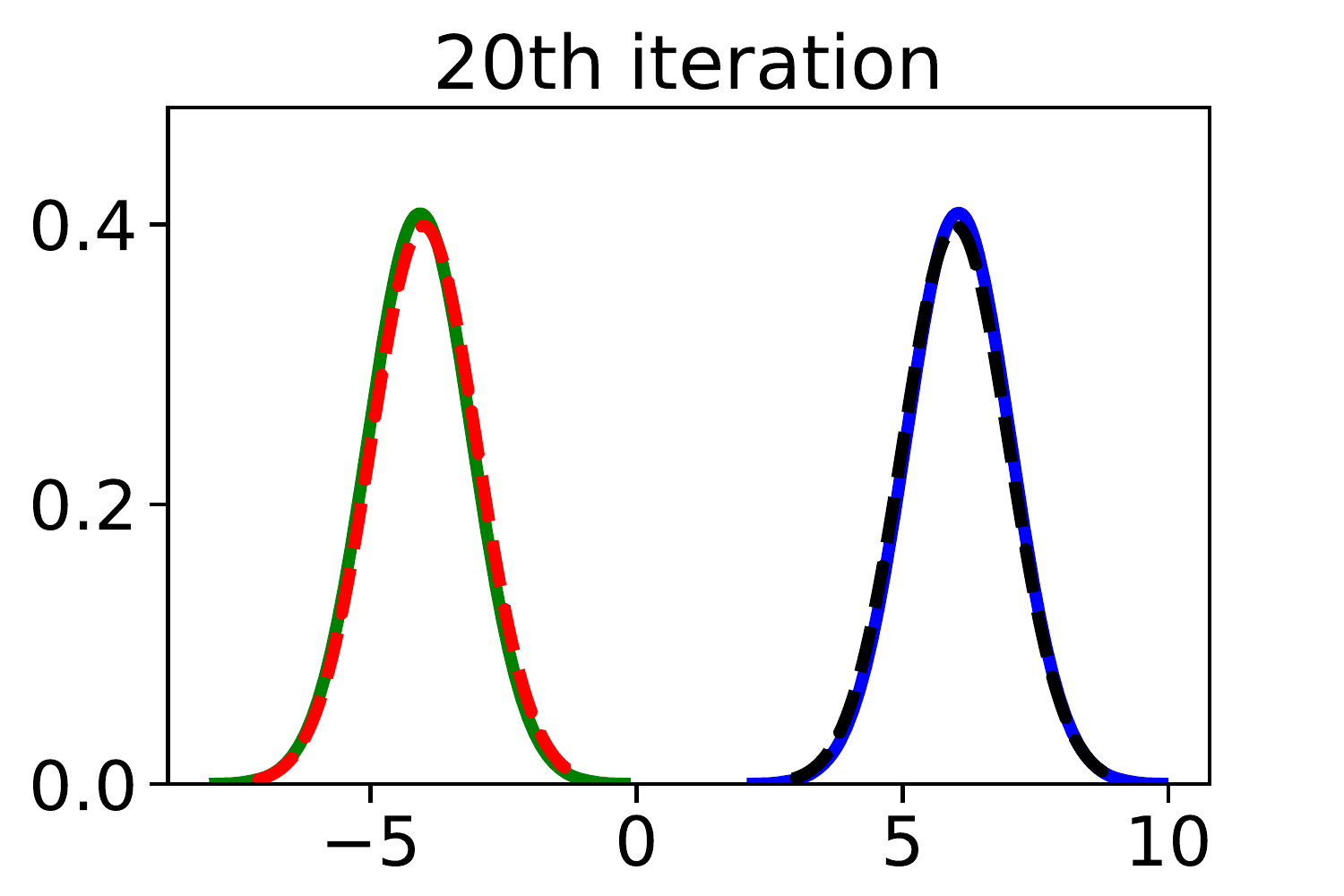} &
\includegraphics[width=.28\textwidth]{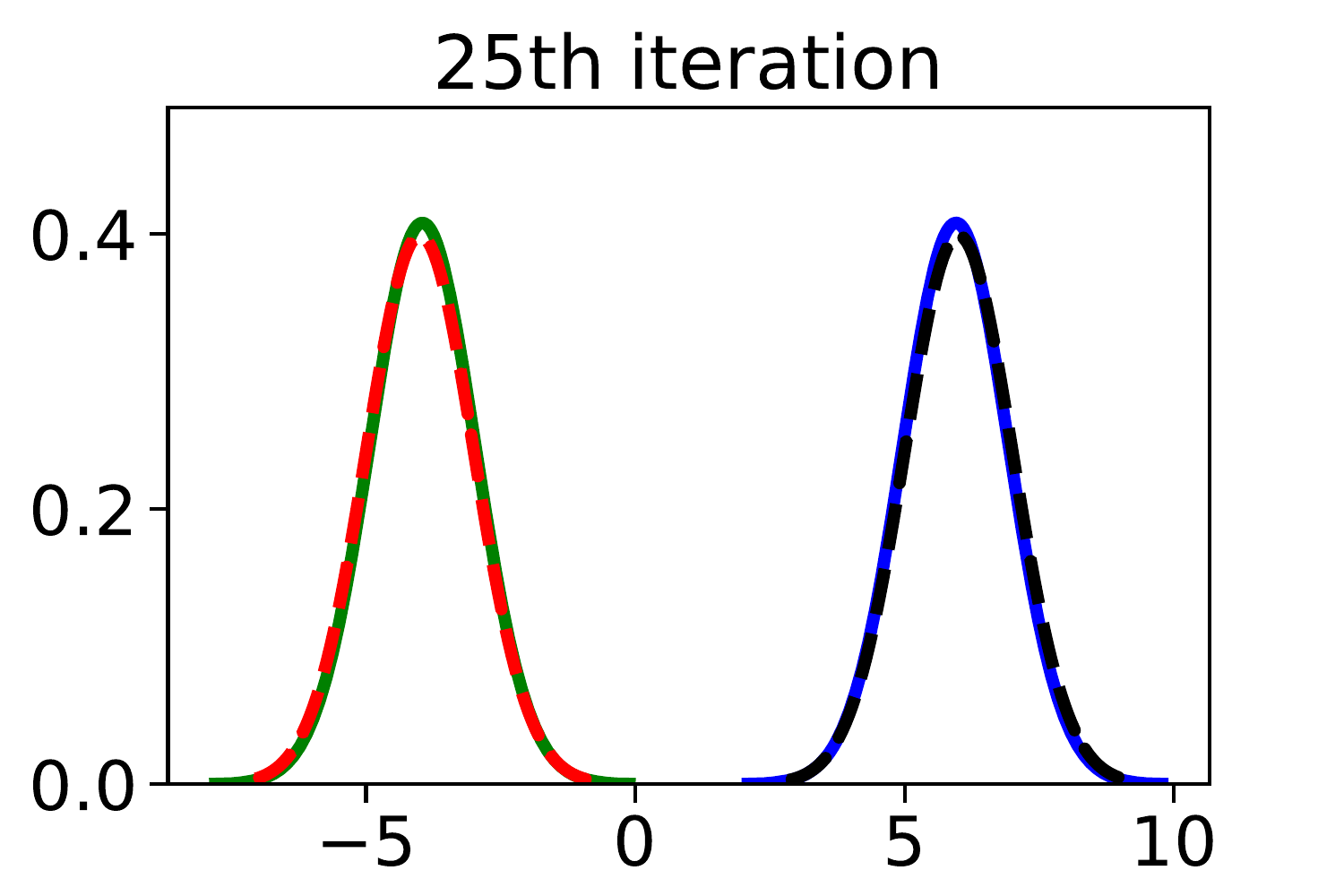}\\
\end{tabular}
}
\caption{Marginal plot for 1D Gaussian example. The red and black dashed lines correspond to two marginal distribution respectively and the solid blue and green lines are the kernel estimated density functions of particles at certain iterations. After first 25 iterations, the particles have matched the marginal distributions very well.}
\label{fig:1dGauss_marg}
\end{figure}

\begin{figure} 
\centering
{\scriptsize 
\begin{tabular}{ccc}
\includegraphics[width=.28\textwidth]{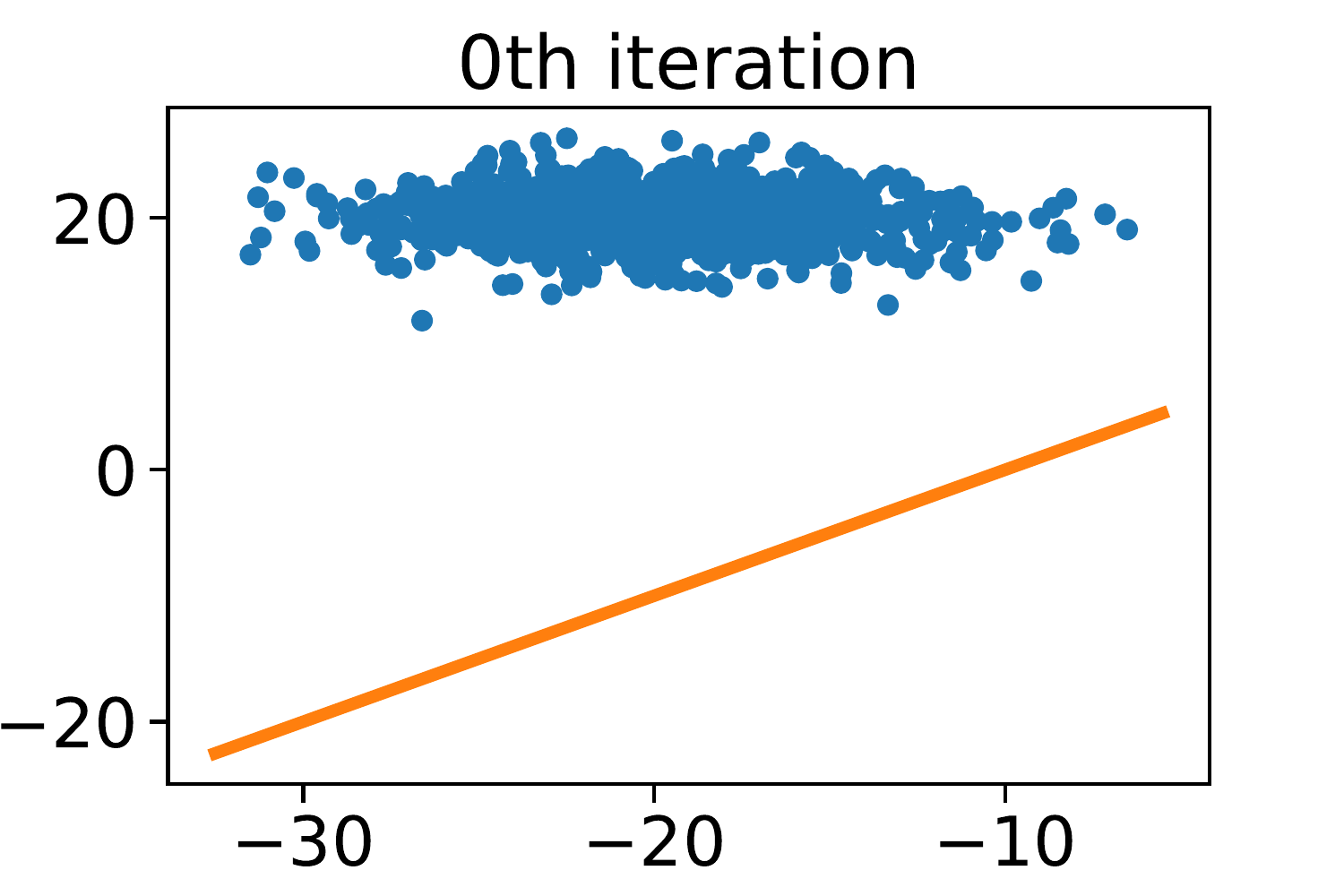} &
\includegraphics[width=.28\textwidth]{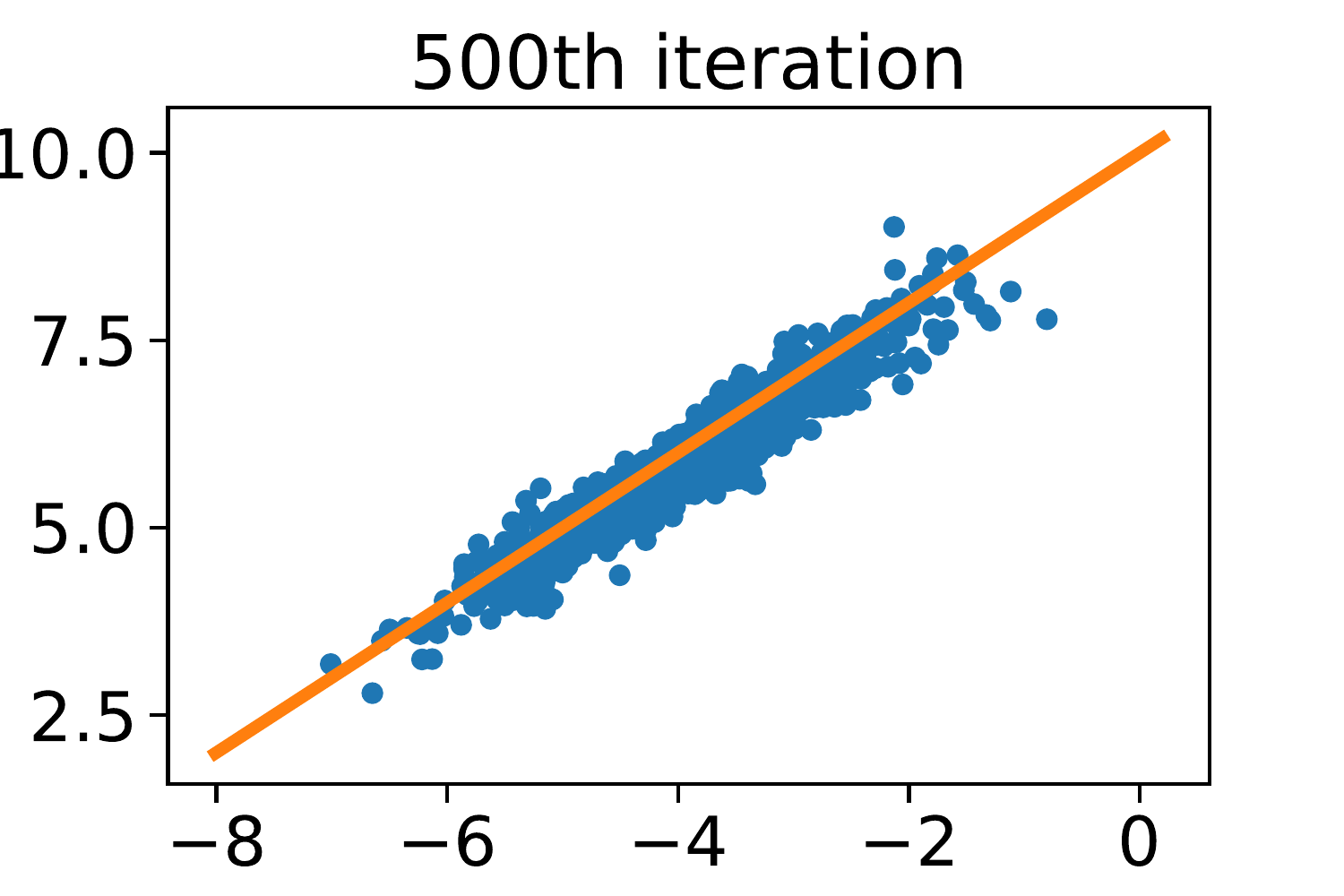} &
\includegraphics[width=.28\textwidth]{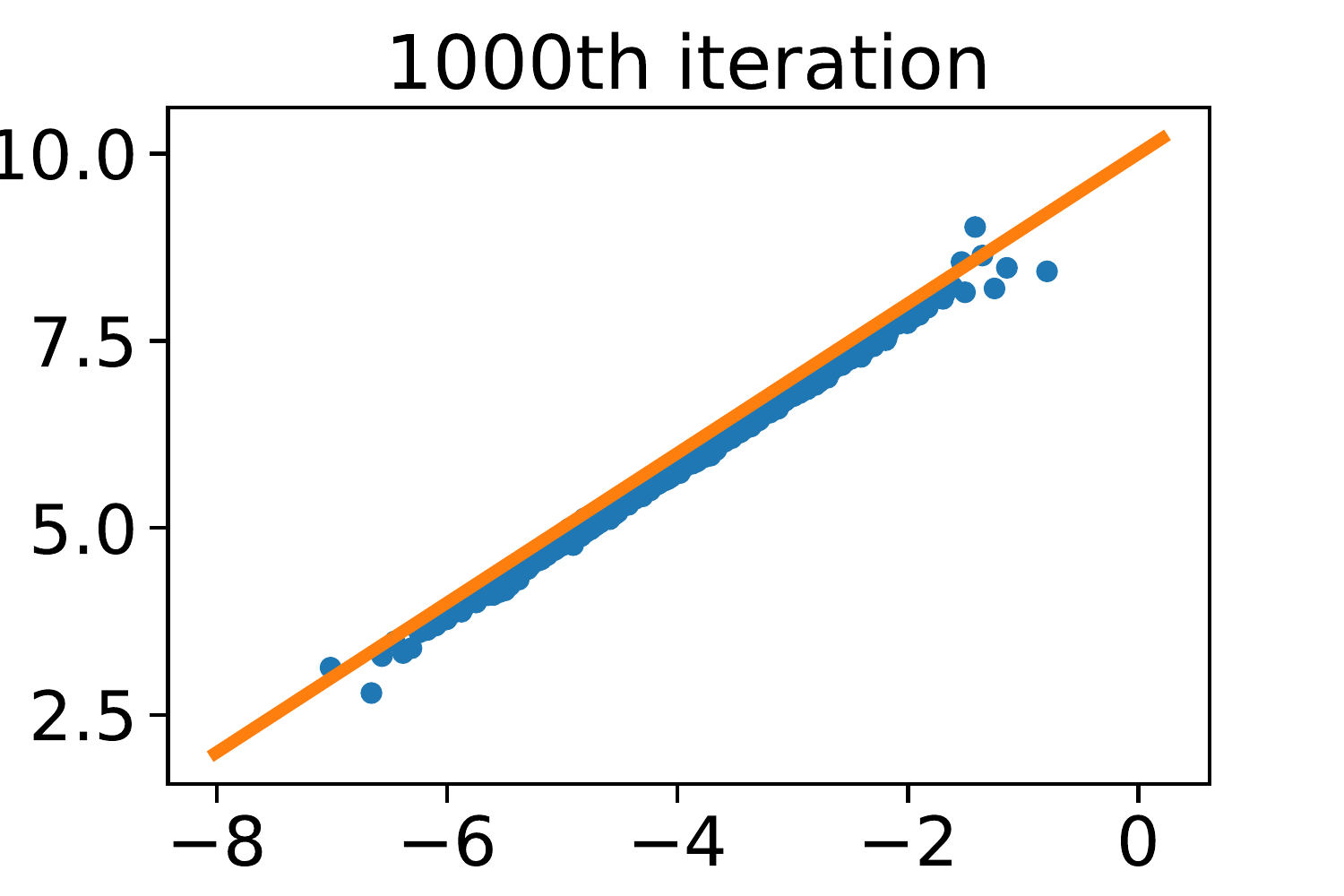} \\
\end{tabular}
}
\caption{The sample approximation for 1D Gaussian example. The orange dash line corresponds to the optimal transport map $T(x)=x+10$.}
\label{fig:1dGauss_map}
\end{figure}


{\bf 1D Gaussian Mixture}
Then we apply the algorithm to two 1D Gaussian mixture $\varrho_1=\frac{1}{2}\mathcal{N}(x;-1,1)+\frac{1}{2}\mathcal{N}(x;1,1), \varrho_2=\frac{1}{2}\mathcal{N}(x;-2,1)+\frac{1}{2}\mathcal{N}(x;2,1)$. For experiment, we set $\lambda=60, \Delta t=0.0004$ and run it with 1000 particles $(X_i,Y_i)$'s for 5000 iterations. We initialize the particles by drawing 2000 i.i.d. sample points from $\mathcal{N}(x;0,2)$ as $X_i$'s and $Y_i$'s. In figure \ref{fig:1dGauss_mix}, we can see the particles still match the marginal distributions well and give a clear approximation for the optimal transport map. 

\begin{figure} 
\centering
{\scriptsize 
\begin{tabular}{cc}
\includegraphics[width=.42\textwidth]{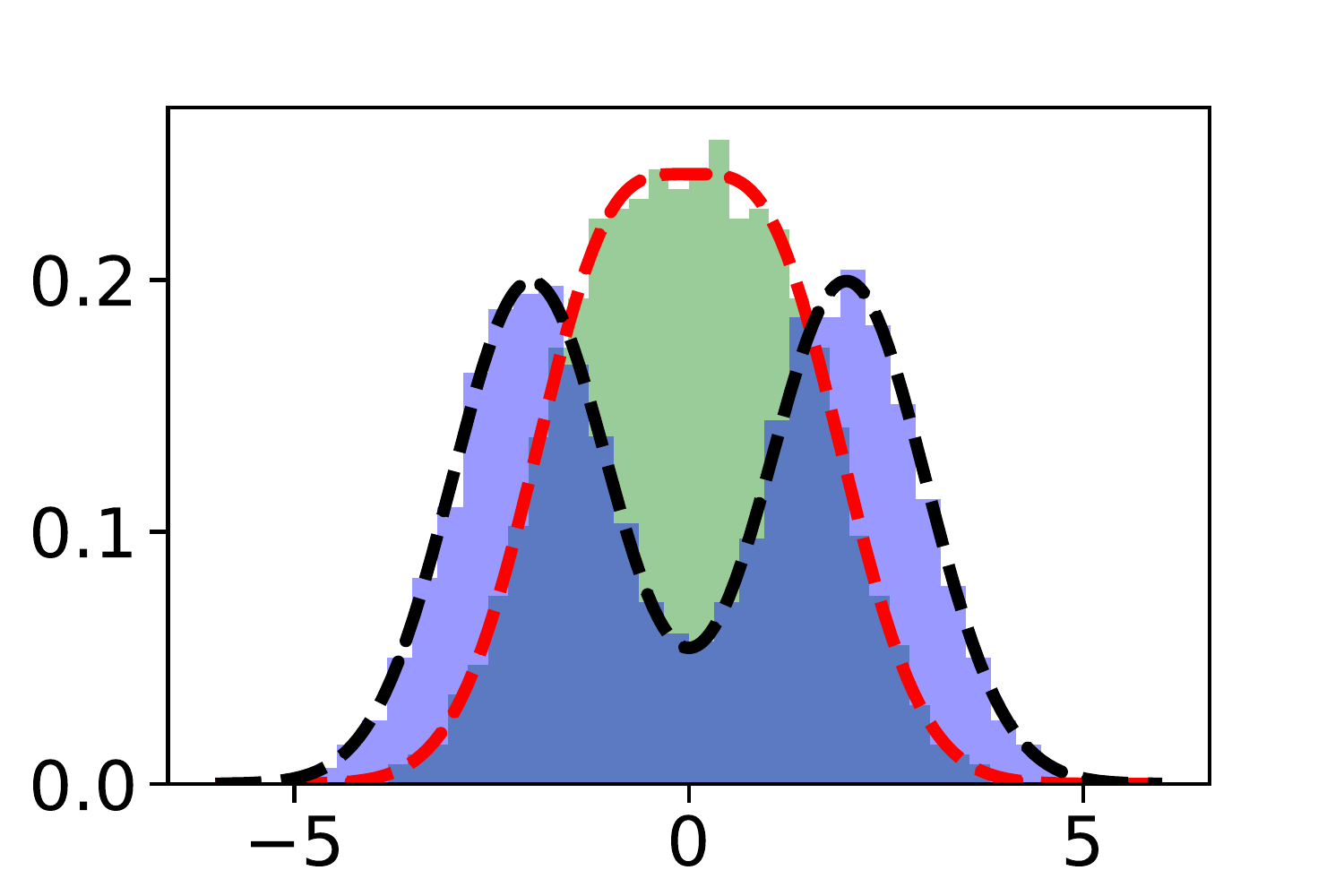} &
\includegraphics[width=.42\textwidth]{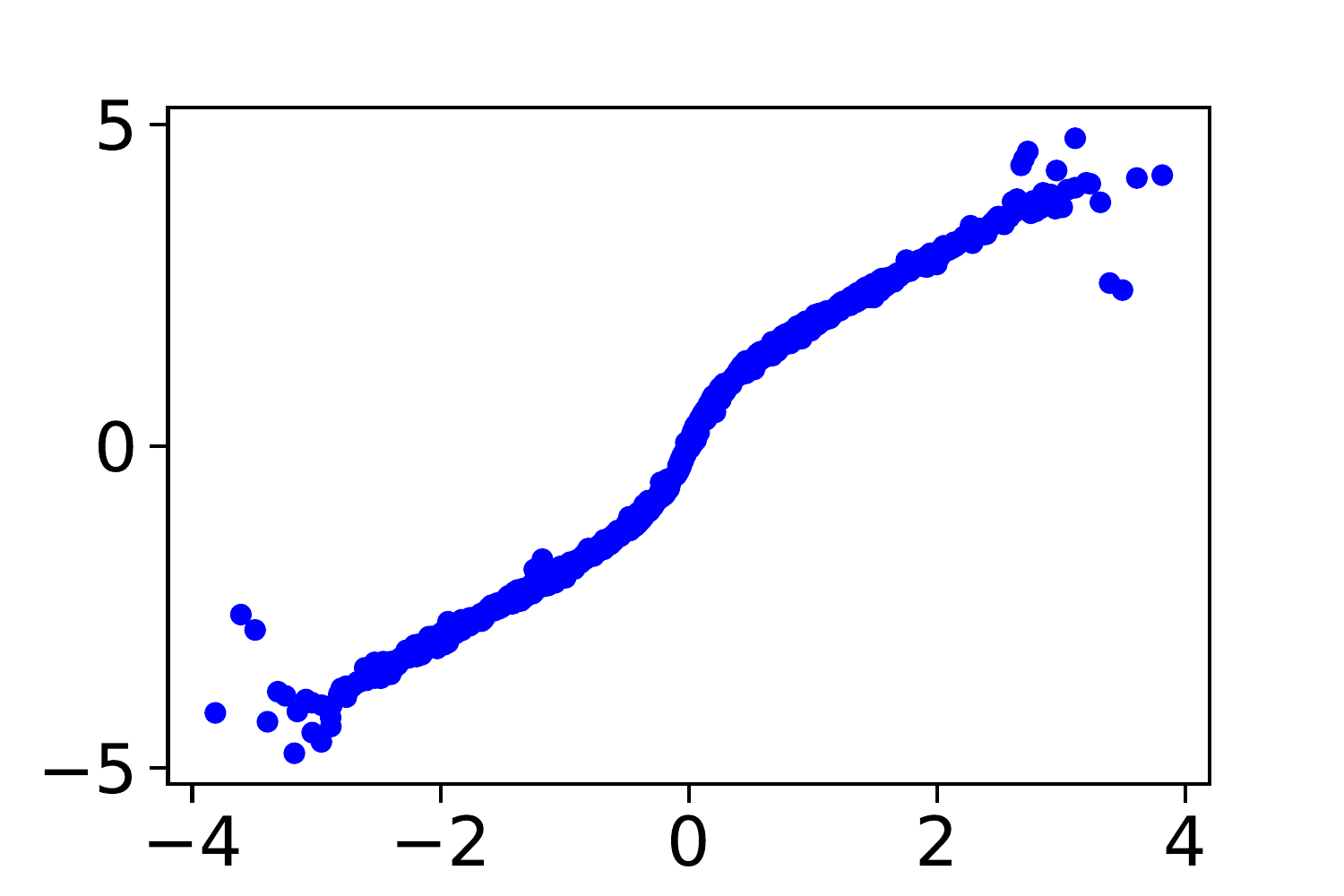} \\
\end{tabular}
}
\vspace{-6pt}
\caption{1D Gaussian mixture. Left. Marginal plot. The dash lines correspond to two marginal distributions. The histogram indicates the distribution of particles after 5000 iterations. Right. Sample approximation for the optimal coupling.}
\label{fig:1dGauss_mix}
\end{figure}

{\bf Synthetic 2D Data} Given two marginals $\mu, \nu$ and cost function $|x-y|^p$, we can get a constant speed geodesic connecting two marginals by defining the curve $\mu_t = (\pi_t)_{\#} \gamma$ where $\gamma$ is the optimal coupling and $\pi_t(x,y) = (1-t)x+t y$. Given two gray scale images, if we normalize the pixel intensity, the image can be treated as a histogram representing a discrete 2D distribution. By applying the RBF kernel, the image can be converted to a continuous distribution as Gaussian mixture. Since our method gives a sample approximation for the optimal coupling, we are able to get the sample approximation of a series of distributions interpolating between two given marginals. In figure \ref{fig:marg}, we plot several simple gray scale images which are converted to continuous probability densities and used as marginals in our experiments. In figures \ref{fig:star_circle_mnist}
, we show two examples of transporting one point cloud image to the other.
\vspace{-6pt}
\begin{figure}[htp]
\centering
{\scriptsize 
\begin{tabular}{cccc}
\includegraphics[width=.14\textwidth]{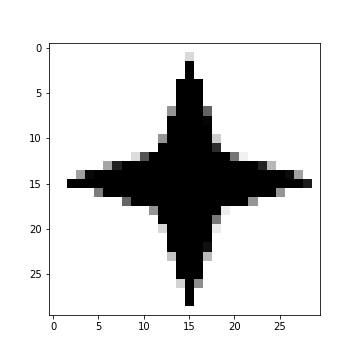} &
\includegraphics[width=.14\textwidth]{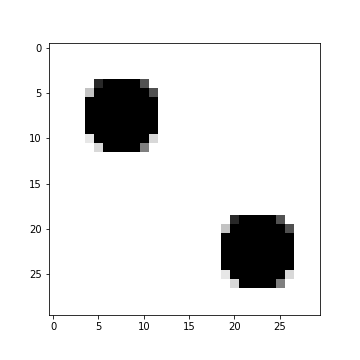} &
\includegraphics[width=.14\textwidth]{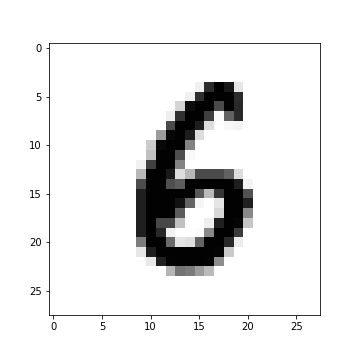} &
\includegraphics[width=.14\textwidth]{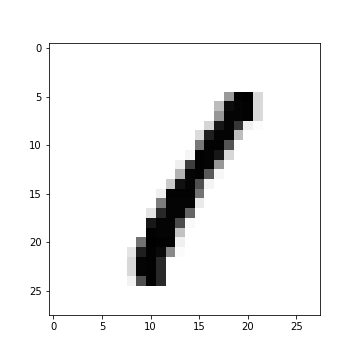} \\
(a) & (b) & (c) & (d) 
\end{tabular}
}
\caption{Some gray scale images.}
\label{fig:marg}
\end{figure}

\begin{figure}[htp]
\centering
{\scriptsize 
\begin{tabular}{cccccc}
\includegraphics[width=.14\textwidth]{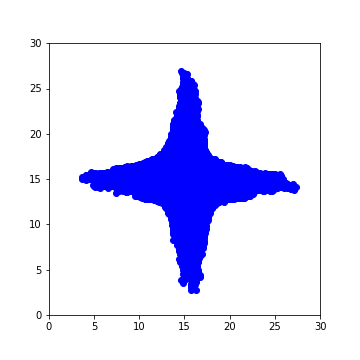} &
\includegraphics[width=.14\textwidth]{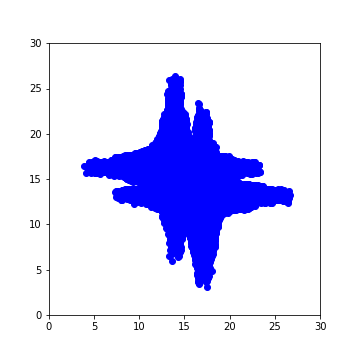} &
\includegraphics[width=.14\textwidth]{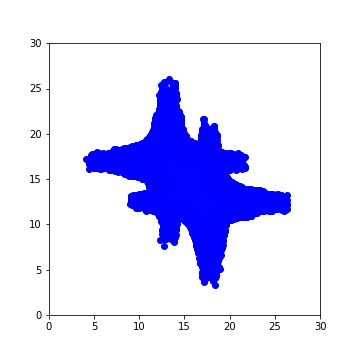} &
\includegraphics[width=.14\textwidth]{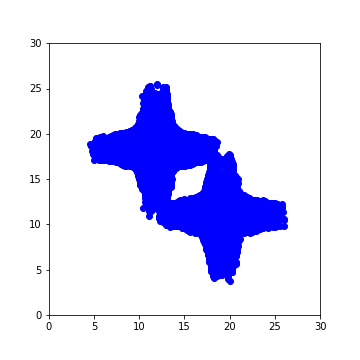} &  
\includegraphics[width=.14\textwidth]{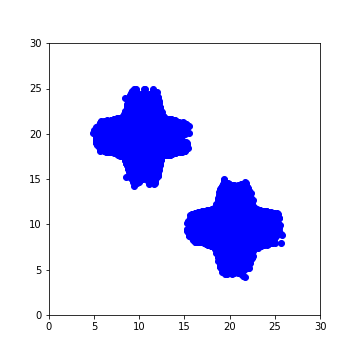} &  
\includegraphics[width=.14\textwidth]{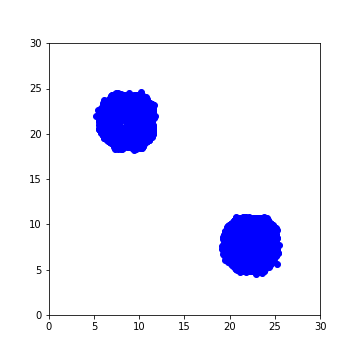} \\
\includegraphics[width=.14\textwidth]{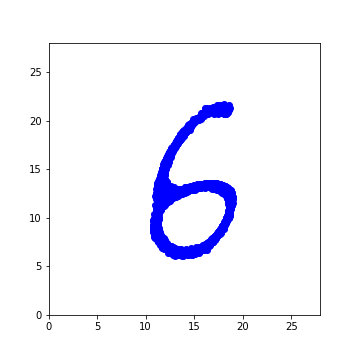} &
\includegraphics[width=.14\textwidth]{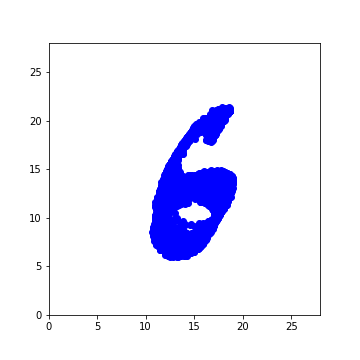} &
\includegraphics[width=.14\textwidth]{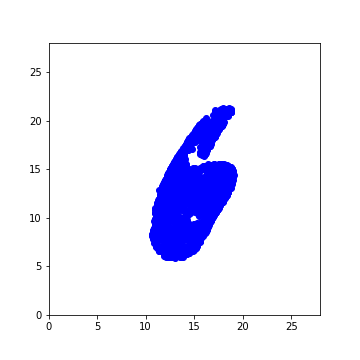} &
\includegraphics[width=.14\textwidth]{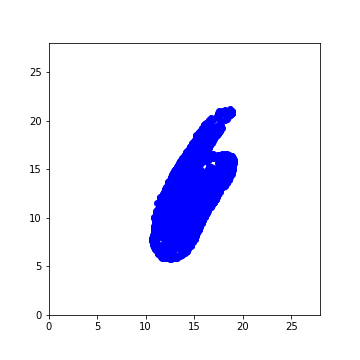} &  
\includegraphics[width=.14\textwidth]{syn2/syn2_5} &  
\includegraphics[width=.14\textwidth]{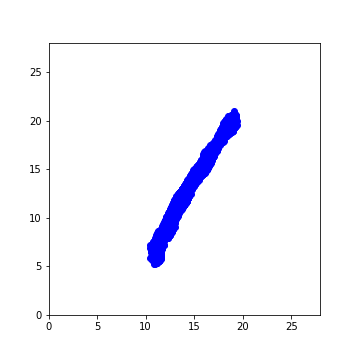}
\end{tabular}
}
\caption{The sample approximation of a sequence of distributions interpolating between the two given distributions: "star" and "disks", MNIST handwritten digits "6" and "1".}
\label{fig:star_circle_mnist}
\end{figure}


{\bf Wasserstein Barycenters}
As we discuss in the previous section, we can numerically solve the Wasserstein barycenter problem using our scheme. Given two Gaussian distributions $\rho_1 = \mathcal{N}(-10,1), \rho_2 = \mathcal{N}(10,1)$, and cost function
\begin{equation*}
    c(x,x_1,x_2) = w_1 \|x-x_1\|^2 + w_2 \|x-x_2\|^2,
\end{equation*} we can compute sample approximation of the barycenter $\bar{\rho}$ of $\rho_1,\rho_2$. We try different weights $[w_1,w_2]=[0.25,0.75],[0.5,0.5],[0.75,0.25]$ to test our algorithm. The experimental results are shown in fig \ref{fig:baryc}. The distribution of the particles corresponding to the barycenter random variable $X_0$ converges to $\mathcal{N}(5,1), \mathcal{N}(0,1), \mathcal{N}(-5,1)$ successfully after 2000 iterations, which demonstrates the accuracy of the algorithm.

\begin{figure}[htp]
\centering
{\scriptsize 
\begin{tabular}{ccc}
\includegraphics[width=.28\textwidth]{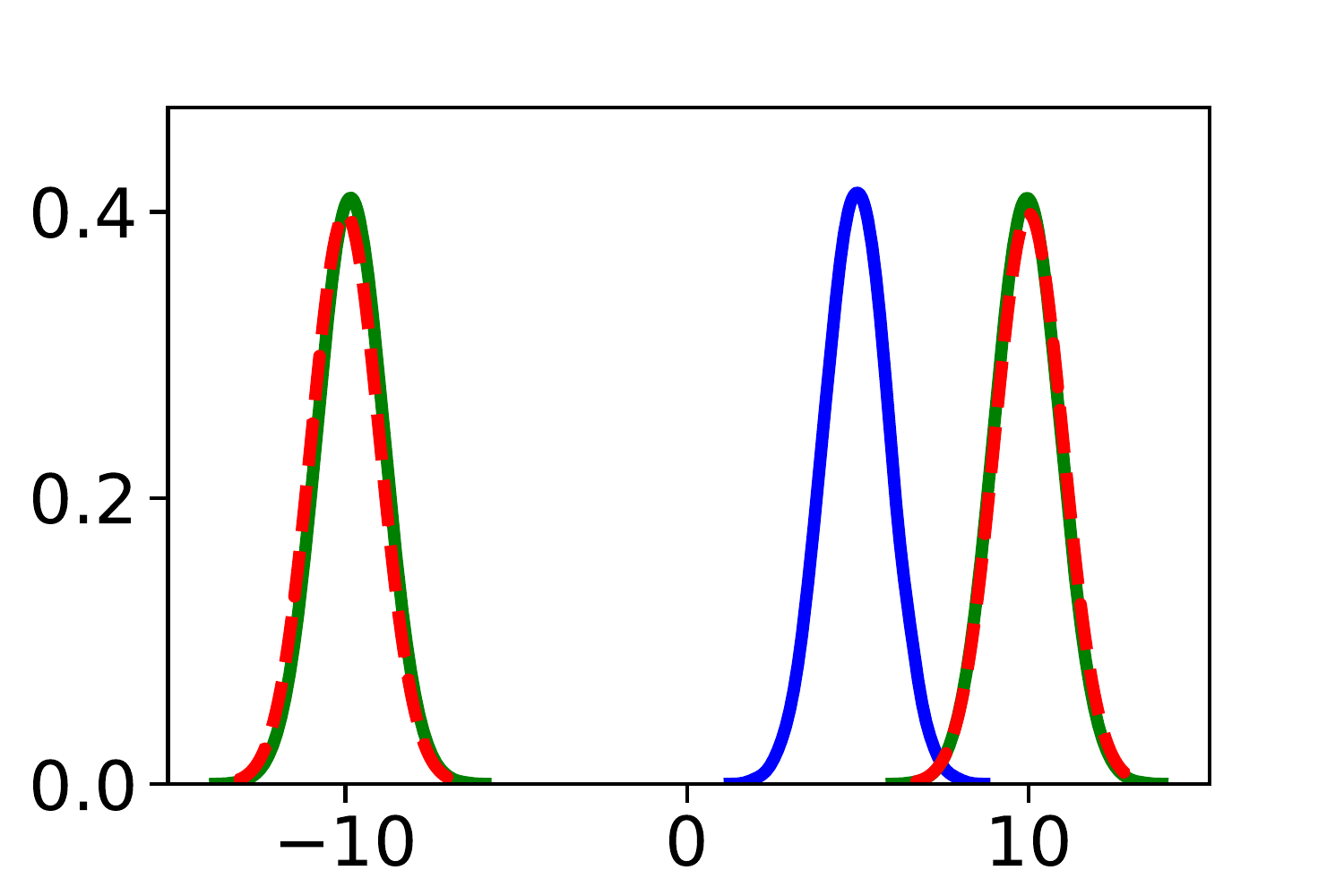} &
\includegraphics[width=.28\textwidth]{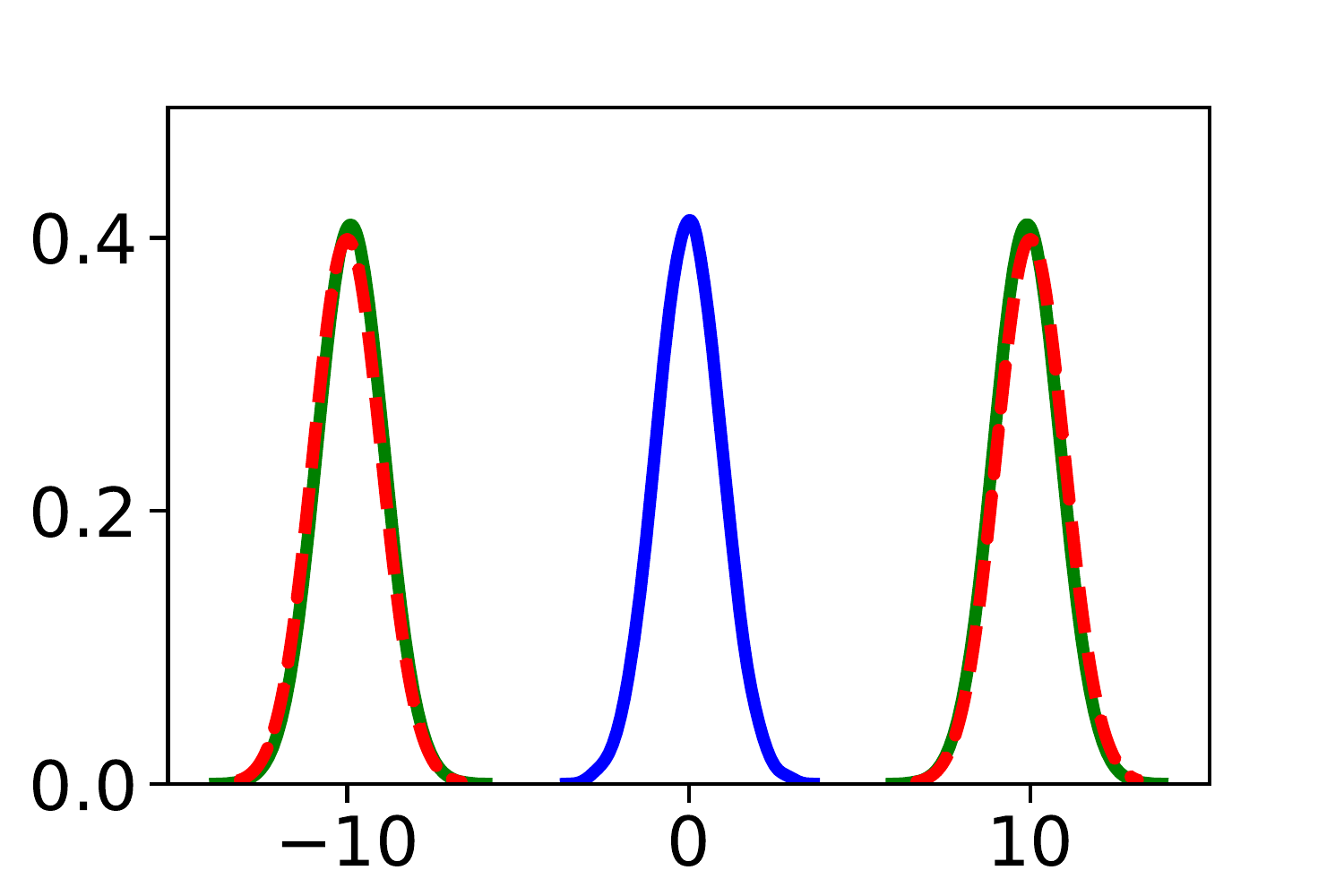} &
\includegraphics[width=.28\textwidth]{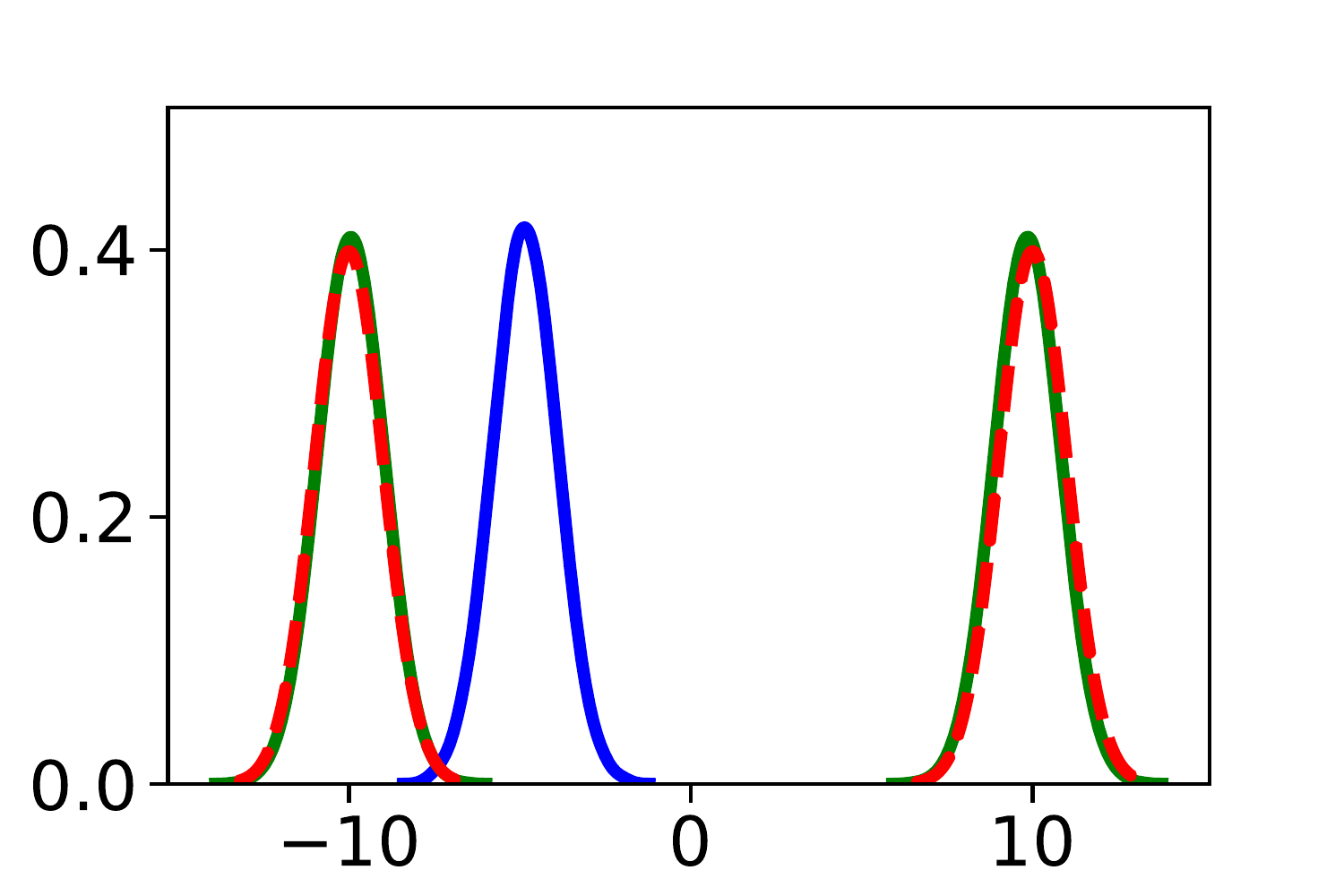} 
\end{tabular}
}
\caption{Density plots for 1D Wasserstein barycenter example.  The red dashed lines correspond to two marginal distributions respectively and the solid green lines are the kernel estimated density functions of the particles $X_1$'s and $X_2$'s. The solid blue line represents the kernel estimated density function of the particles corresponding to the barycenter. Left. $[w_1,w_2]=[0.25,0.75]$. Middle. $[w_1,w_2]=[0.5,0.5]$. Right. $[w_1,w_2]=[0.75,0.25]$.}
\label{fig:baryc}
\end{figure}

\section{Conclusion}
We propose the constrained Entropy Transport problem \eqref{constrained-ET} and study its theoretical properties. We discover that the optimal distribution of \eqref{constrained-ET} can be treated as an approximation to the optimal plan of the Optimal Transport problem \eqref{OT} in the sense of $\Gamma$-convergence . We also construct the Wasserstein gradient flow of the Entropy Transport functional. Based on that, we propose an innovative algorithm which iteratively evolves a particle system to compute for the sample-wised optimal distribution to the constrained Entropy Transport problem \eqref{constrained-ET}. 

For future work, on theoretical aspect, we will mainly concentrate on the quantitative study of the discrepancy between $\gamma_{cET}$ and $\gamma_{OT}$ and the analysis of displacement convexity of functional $\mathcal{E}_{\Lambda,\rm{KL}}(\cdot|\mu,\nu)$. On numerical aspect, we will focus more on producing further examples in higher dimensional space and finding potential applications of our method to different areas of machine learning research. 
%
%
%
\bibliographystyle{splncs04}
\bibliography{example_paper}

\begin{thebibliography}{10}
\providecommand{\url}[1]{\texttt{#1}}
\providecommand{\urlprefix}{URL }
\providecommand{\doi}[1]{https://doi.org/#1}

\bibitem{agueh2011barycenters}
Agueh, M., Carlier, G.: Barycenters in the wasserstein space. SIAM Journal on
  Mathematical Analysis  \textbf{43}(2),  904--924 (2011)

\bibitem{ambrosio2008gradient}
Ambrosio, L., Gigli, N., Savar{\'e}, G.: Gradient flows: in metric spaces and
  in the space of probability measures. Springer Science \& Business Media
  (2008)

\bibitem{arjovsky2017wasserstein}
Arjovsky, M., Chintala, S., Bottou, L.: Wasserstein gan. arXiv preprint
  arXiv:1701.07875  (2017)

\bibitem{benamou2000computational}
Benamou, J.D., Brenier, Y.: A computational fluid mechanics solution to the
  monge-kantorovich mass transfer problem. Numerische Mathematik
  \textbf{84}(3),  375--393 (2000)

\bibitem{benamou2014numerical}
Benamou, J.D., Froese, B.D., Oberman, A.M.: Numerical solution of the optimal
  transportation problem using the monge--amp{\`e}re equation. Journal of
  Computational Physics  \textbf{260},  107--126 (2014)

\bibitem{billingsley2013convergence}
Billingsley, P.: Convergence of probability measures. John Wiley \& Sons (2013)

\bibitem{braides2006handbook}
Braides, A.: A handbook of $\gamma$-convergence. In: Handbook of Differential
  Equations: stationary partial differential equations, vol.~3, pp. 101--213.
  Elsevier (2006)

\bibitem{brenier1991polar}
Brenier, Y.: Polar factorization and monotone rearrangement of vector-valued
  functions. Communications on pure and applied mathematics  \textbf{44}(4),
  375--417 (1991)

\bibitem{carrillo2019aggregation}
Carrillo, J.A., Craig, K., Yao, Y.: Aggregation-diffusion equations: dynamics,
  asymptotics, and singular limits. In: Active Particles, Volume 2, pp.
  65--108. Springer (2019)

\bibitem{carrillo2019blob}
Carrillo, J.A., Craig, K., Patacchini, F.S.: A blob method for diffusion.
  Calculus of Variations and Partial Differential Equations  \textbf{58}(2),
  ~53 (2019)

\bibitem{chen2018unified}
Chen, C., Zhang, R., Wang, W., Li, B., Chen, L.: A unified
  particle-optimization framework for scalable bayesian sampling. arXiv
  preprint arXiv:1805.11659  (2018)

\bibitem{chizat2018unbalanced}
Chizat, L., Peyr{\'e}, G., Schmitzer, B., Vialard, F.X.: Unbalanced optimal
  transport: Dynamic and kantorovich formulations. Journal of Functional
  Analysis  \textbf{274}(11),  3090--3123 (2018)

\bibitem{cuturi2013sinkhorn}
Cuturi, M.: Sinkhorn distances: Lightspeed computation of optimal transport.
  In: Advances in neural information processing systems. pp. 2292--2300 (2013)

\bibitem{cuturi2014fast}
Cuturi, M., Doucet, A.: Fast computation of wasserstein barycenters. In:
  International conference on machine learning. pp. 685--693. PMLR (2014)

\bibitem{daaloul2021sampling}
Daaloul, C., Gouic, T.L., Liandrat, J., Tournus, M.: Sampling from the
  wasserstein barycenter. arXiv preprint arXiv:2105.01706  (2021)

\bibitem{dal2012introduction}
Dal~Maso, G.: An introduction to $\Gamma$-convergence, vol.~8. Springer Science
  \& Business Media (2012)

\bibitem{gangbo1998optimal}
Gangbo, W., {\'S}wiech, A.: Optimal maps for the multidimensional
  monge-kantorovich problem. Communications on Pure and Applied Mathematics: A
  Journal Issued by the Courant Institute of Mathematical Sciences
  \textbf{51}(1),  23--45 (1998)

\bibitem{jin2020random}
Jin, S., Li, L., Liu, J.G.: Random batch methods (rbm) for interacting particle
  systems. Journal of Computational Physics  \textbf{400},  108877 (2020)

\bibitem{jordan1998variational}
Jordan, R., Kinderlehrer, D., Otto, F.: The variational formulation of the
  fokker--planck equation. SIAM journal on mathematical analysis
  \textbf{29}(1),  1--17 (1998)

\bibitem{kantorovich1942translation}
Kantorovich, L.: On translation of mass (in russian), c r. In: Doklady. Acad.
  Sci. USSR. vol.~37, pp. 199--201 (1942)

\bibitem{korotin2019wasserstein}
Korotin, A., Egiazarian, V., Asadulaev, A., Safin, A., Burnaev, E.:
  Wasserstein-2 generative networks. arXiv preprint arXiv:1909.13082  (2019)

\bibitem{kullback1951information}
Kullback, S., Leibler, R.A.: On information and sufficiency. The annals of
  mathematical statistics  \textbf{22}(1),  79--86 (1951)

\bibitem{Lafferty}
Lafferty, J.D.: The {{Density Manifold}} and {{Configuration Space
  Quantization}}. Transactions of the American Mathematical Society
  \textbf{305}(2),  699--741 (1988)

\bibitem{li2018parallel}
Li, W., Ryu, E.K., Osher, S., Yin, W., Gangbo, W.: A parallel method for earth
  mover’s distance. Journal of Scientific Computing  \textbf{75}(1),
  182--197 (2018)

\bibitem{liero2018optimal}
Liero, M., Mielke, A., Savar{\'e}, G.: Optimal entropy-transport problems and a
  new hellinger--kantorovich distance between positive measures. Inventiones
  mathematicae  \textbf{211}(3),  969--1117 (2018)

\bibitem{makkuva2020optimal}
Makkuva, A., Taghvaei, A., Oh, S., Lee, J.: Optimal transport mapping via input
  convex neural networks. In: International Conference on Machine Learning. pp.
  6672--6681. PMLR (2020)

\bibitem{monge1781memoire}
Monge, G.: M{\'e}moire sur la th{\'e}orie des d{\'e}blais et des remblais.
  Histoire de l'Acad{\'e}mie Royale des Sciences de Paris  (1781)

\bibitem{oberman2015efficient}
Oberman, A.M., Ruan, Y.: An efficient linear programming method for optimal
  transportation. arXiv preprint arXiv:1509.03668  (2015)

\bibitem{otto2001}
Otto, F.: The {{Geometry}} of {{Dissipative Evolution Equations}}: {{The Porous
  Medium Equation}}. Communications in Partial Differential Equations
  \textbf{26}(1-2),  101--174 (2001)

\bibitem{parzen1962estimation}
Parzen, E.: On estimation of a probability density function and mode. The
  annals of mathematical statistics  \textbf{33}(3),  1065--1076 (1962)

\bibitem{pele2009fast}
Pele, O., Werman, M.: Fast and robust earth mover's distances. In: 2009 IEEE
  12th International Conference on Computer Vision. pp. 460--467. IEEE (2009)

\bibitem{peyre2019computational}
Peyr{\'e}, G., Cuturi, M., et~al.: Computational optimal transport. Foundations
  and Trends{\textregistered} in Machine Learning  \textbf{11}(5-6),  355--607
  (2019)

\bibitem{ruthotto2020machine}
Ruthotto, L., Osher, S.J., Li, W., Nurbekyan, L., Fung, S.W.: A machine
  learning framework for solving high-dimensional mean field game and mean
  field control problems. Proceedings of the National Academy of Sciences
  \textbf{117}(17),  9183--9193 (2020)

\bibitem{schmitzer2016sparse}
Schmitzer, B.: A sparse multiscale algorithm for dense optimal transport.
  Journal of Mathematical Imaging and Vision  \textbf{56}(2),  238--259 (2016)

\bibitem{seguy2017large}
Seguy, V., Damodaran, B.B., Flamary, R., Courty, N., Rolet, A., Blondel, M.:
  Large-scale optimal transport and mapping estimation. arXiv preprint
  arXiv:1711.02283  (2017)

\bibitem{villani2008optimal}
Villani, C.: Optimal transport: old and new, vol.~338. Springer Science \&
  Business Media (2008)

\bibitem{walsh2017general}
Walsh~III, J.D., Dieci, L.: General auction method for real-valued optimal
  transport. arXiv preprint arXiv:1705.06379  (2017)

\end{thebibliography}
%





\newpage

\onecolumn

\appendix
\section{Appendix A}\label{App a}

In this appendix, we present several important theorems regarding Entropy Transport problem \eqref{general ET} and our proposed constrained Entropy Transport problem \eqref{constrained-ET}.
\subsection{Entropy Transport problem}\label{App a1}

Let us recall the Entropy Transport problem:\\
For $\mu,\nu\in\mathcal{M}(\mathbb{R}^d)$ and $\gamma\in\mathcal{M}(\mathbb{R}^d\times\mathbb{R}^d)$, we denote $\gamma_1 = \pi_{1 \#}\gamma, \gamma_2 = \pi_{2 \#}\gamma$. Here $\pi_1:\mathbb{R}^d\times\mathbb{R}^d\rightarrow \mathbb{R}^d$, is the projection onto the first coordinate: $\pi_1(x,y)=x$; and $\pi_2$ is the projection onto the second coordinate. We consider the functional:
\begin{equation*}
   \mathcal{E}(\gamma|\mu,\nu)=  \iint_{\mathbb{R}^d\times\mathbb{R}^d}c(x,y)d\gamma(x,y)+D({\pi_1}_{\#}\gamma \|\mu)+D({\pi_2}_{\#}\gamma \|\nu).
\end{equation*}
Here $c:\mathbb{R}^d\times\mathbb{R}^d\rightarrow [0,+\infty]$ is a lower semicontinuous cost function.
$D(\cdot\|\cdot):\mathcal{M}(\mathbb{R}^d)\times\mathcal{M}(\mathbb{R}^d)\rightarrow \mathbb{R}$ is the divergence functional defined as:
\begin{equation*}
  D(\mu\|\nu) = 
  \begin{cases}
  \int_{\mathbb{R}^d} F\left(\frac{d\mu}{d\nu}\right)d\nu \quad \textrm{if}~\mu\ll\nu\\
  +\infty \quad \textrm{otherwise}
  \end{cases}
\end{equation*}
Here $F:[0,+\infty)\rightarrow [0,+\infty]$ is some convex function and there exists at least one $s>0$ such that $F(s)<+\infty$. 

The \textbf{(general) Entropy Transport problem} is:
\begin{equation*}
 \inf\limits_{\gamma\in\mathcal{M}(\mathbb{R}^d\times\mathbb{R}^d)}\left\{ \mathcal{E}(\gamma|\mu,\nu) \right\}. 
\end{equation*}
The following theorem shows $\mathcal{E}(\cdot|\mu,\nu)$ is convex on $\mathcal{P}(\mathbb{R}^d\times\mathbb{R}^d)$:
\begin{theorem}
Under the previous assumptions on $c$ and $F$, for any $\gamma_a,\gamma_b\in\mathcal{M}(\mathbb{R}^d\times\mathbb{R}^d)$ and $0\leq t\leq 1$, we have:
\begin{equation*}
\mathcal{E}(t\gamma_a+(1-t)\gamma_b)\leq t\mathcal{E}(\gamma_a)+(1-t)\mathcal{E}(\gamma_b).
\end{equation*}
\end{theorem}
\begin{proof}
   First we prove the result under following conditions:
  \begin{equation}
   \pi_{1 \#}\gamma_a\ll\mu,~\pi_{1 \#}\gamma_b\ll\mu,~\pi_{2 \#}\gamma_a\ll\nu,~\pi_{2\#}\gamma_b\ll\nu, \label{4 cond}
  \end{equation}
  we have $\pi_{1 \#}(t\gamma_a + (1-t)\gamma_b )\ll\mu$ and $\pi_{2 \#}( t \gamma_a +(1-t)\gamma_b   )\ll\nu$. We thus have
  \begin{equation*}
    F\left(\frac{d\pi_{1 \#}(t\gamma_a+(1-t)\gamma_b)}{d\mu}\right) = F\left( t \frac{d\pi_{1 \#}\gamma_a}{ d \mu} + (1-t)\frac{d\pi_{1\#}\gamma_b}{d\mu}   \right)\leq t F\left(\frac{d\pi_{1 \#}\gamma_a}{d\mu}\right) + (1-t)F\left(\frac{d\pi_{1 \#}\gamma_b}{ d \mu}\right). 
  \end{equation*}
 As a result, one can prove $D(\pi_{1 \#}(t\gamma_a+(1-t)\gamma_b)\|\mu)\leq t D(\pi_{1 \#}\gamma_a\|\mu)+(1-t)D(\pi_{1 \#}\gamma_b\|\mu  )$.
 We can prove similar inequality for the other side of marginal. And then it is not hard to verify $\mathcal{E}(t\gamma_a+(1-t)\gamma_b)\leq t\mathcal{E}(\gamma_a)+(1-t)\mathcal{E}(\gamma_b).$\\
 Now when any one of the four conditions in \eqref{4 cond} is not satisfied, the right hand side of the inequality is $+\infty$, thus the inequality still holds.
\end{proof}

The following theorem gives sufficient conditions for the existence and uniqueness of the solution to the Entropy Transport problem \eqref{general ET}: 
\begin{theorem}
  We consider problem \eqref{general ET} involving the entropy transport functional defined in \eqref{general ET functional}. Suppose that the cost $c$ and $F$ satisfy the previous assumptions. We further assume that there exists at least one $\gamma\in\mathcal{M}(\mathbb{R}^d\times\mathbb{R}^d)$ such that $\mathcal{E}(\gamma|\mu,\nu)<+\infty$. Then the problem \eqref{general ET} admits at least one optimal solution.\\
  If we further assume $c(x,y)=h(x-y)$ with strictly convex $h:\mathbb{R}^d\rightarrow[0,+\infty)$; $F$ is strictly convex, and is superlinear, i.e. $\lim_{s\rightarrow+\infty}\frac{F(s)}{s}=+\infty$; distribution $\mu$ has density function, i.e. $\mu\ll\mathscr{L}^d$ where $\mathscr{L}^d$ as the Lebesgue measure on $\mathbb{R}^d$. 
  Under these further assumptions, there exists unique optimal solution to the problem.
\end{theorem}
This theorem is a direct summarize of Theorem 3.3; Corollary 3.6 and Example 3.7 of \cite{liero2018optimal}.

\subsection{Constrained Entropy Transport problem}\label{App a2}
We consider the following functional:
\begin{equation*}
  \mathcal{E}_{\Lambda,\rm{KL}}(\gamma|\mu,\nu) = \iint_{\mathbb{R}^d\times\mathbb{R}^d} c(x,y)~d\gamma(x,y) + \Lambda D_{\rm{KL}}(\pi_{1 \#} \gamma \| \mu ) + \Lambda D_{\rm{KL}}(\pi_{2 \#} \gamma \| \nu ).
\end{equation*}
with assumptions:
\begin{align*}
   (A).~& c(x,y) = h(x-y)~\textrm{with}~ h ~\textrm{a strictly convex function}.\\
   \quad (B). ~& \mu,\nu\in\mathcal{P}(\mathbb{R}^d)~\rm{and}~   \mu\ll\mathscr{L}^d, \nu\ll\mathscr{L}^d ~ 
\end{align*}
We consider the following \textbf{constrained Entropy Transport problem}:
\begin{equation*}
  \inf_{\gamma\in\mathcal{P}(\mathbb{R}^d\times\mathbb{R}^d)}\left\{ \mathcal{E}_{\Lambda,\rm{KL}}(\gamma |\mu , \nu)  \right\}.
\end{equation*}
By choosing $\gamma=\mu\otimes \nu$, i.e. choose $\gamma$ as the direct product of $\mu,\nu$, we have $\mathcal{E}_{\Lambda,\rm{KL}}(\mu\otimes\nu |\mu , \nu)=\iint cd\mu\otimes\nu\geq0$. One can prove that $\underset{  \gamma\in\mathcal{P}(\mathbb{R}^d\times\mathbb{R}^d) }{\inf   }\mathcal{E}_{\Lambda,\rm{KL}}(\gamma |\mu , \nu) \geq 0$. Thus the infimum value is finite and bounded from below by $0$. Let us now denote:
\begin{equation}
 \mathcal{E}_{\min } = \underset{  \gamma\in\mathcal{P}(\mathbb{R}^d\times\mathbb{R}^d) }{\inf   }\mathcal{E}_{\Lambda,\rm{KL}}(\gamma |\mu , \nu).\label{def Epsilon_min}
\end{equation}
The following theorem shows the existence of the optimal solution to problem \eqref{constrained-ET}. It also describes the relationship between the solution of constrained Entropy Transport problem and the solution of general Entropy Transport problem:
\begin{theorem} \label{rel}
  Suppose  $\tilde{\gamma}$ is the solution to original entropy transport problem \eqref{general ET}. Then we have $\tilde{\gamma}=Z\gamma$, here $Z = e^{-\frac{\mathcal{E}_{\rm{min}}}{2\Lambda}       } $ and $\gamma\in\mathcal{P}(\mathbb{R}^d\times\mathbb{R}^d)$ is the solution to constrained Entropy Transport problem \eqref{constrained-ET}.
\end{theorem}
\begin{proof}
  For any $\tilde{\sigma}\in\mathcal{M}(\mathbb{R}^d\times\mathbb{R}^d)$, we can write it as:
  \begin{equation*}
     \tilde{\sigma}=M\sigma
  \end{equation*}
  with $M = \tilde{\sigma}(\mathbb{R}^d\times\mathbb{R}^d)$ and $\sigma\in\mathcal{P}(\mathbb{R}^d\times\mathbb{R}^d)$. Now one can write $\mathcal{E}_{\Lambda,\rm{KL}}(\tilde{\sigma} | \mu,\nu)$ as:
  \begin{align*}
  \mathcal{E}_{\Lambda,\rm{KL}}(\tilde{\sigma}|\mu,\nu)=&\iint c(x,y)~d(M\sigma) + \Lambda \int \left( \frac{d \pi_{1\#}(M \sigma)}{d\mu}\log\left(\frac{d \pi_{1\#} (M\sigma)}{d\mu}\right)-\frac{d\pi_{1\#}(M\sigma)}{d\mu}+1\right)d\mu\\ 
  &+\Lambda \int \left( \frac{d\pi_{2\#}(M\sigma)}{d\nu}\log\left(\frac{d \pi_{2\#}( M\sigma )}{d\nu}\right)-\frac{d\pi_{2\#}(M\sigma)}{d\nu}+1\right)d\nu \\
  =& M \mathcal{E}_{\Lambda,\rm{KL}}(\sigma|\mu,\nu) + 2\Lambda(M\log M -M)+2\Lambda. 
  \end{align*}
  The optimization problem \eqref{general ET} on $\mathcal{M}(\mathbb{R}^d\times\mathbb{R}^d)$ can now be formulated as:
  \begin{equation*}
    \inf_{\sigma\in\mathcal{P}(\mathbb{R}^d\times\mathbb{R}^d)}\min_{M\geq 0} ~\left\{ M\mathcal{E}_{\Lambda,\rm{KL}}(\sigma|\mu,\nu)+2\Lambda (M\log M -M) + 2\Lambda \right\}.
  \end{equation*}
  It is not hard to verify that when $\sigma$ is fixed, we denote $E(\sigma) = \mathcal{E}_{\Lambda,\rm{KL}}(\sigma|\mu,\nu)$ for shorthand. 
  Then the minimum value of $ME(\sigma)+2\Lambda (M\log M-1)+2\Lambda$ ($M\geq 0$) is achieved at 
  $M = e^{-\frac{ E(\sigma)}{2\Lambda}}$ and the minimum value is $2\Lambda ( 1- e^{-\frac{E(\sigma)}{2\Lambda}})$. 
  Recall definition \eqref{def Epsilon_min}, we have:
  \begin{align*}
  &\inf_{\sigma\in\mathcal{P}(\mathbb{R}^d\times\mathbb{R}^d)}\min_{M\geq 0} ~\left\{ M\mathcal{E}_{\Lambda,\rm{KL}}(\sigma|\mu,\nu)+2\Lambda( M\log M-M)+2\Lambda \right\}\\
  =& \inf_{\sigma\in\mathcal{P}(\mathbb{R}^d\times\mathbb{R}^d)} \left\{2\Lambda(1- e^{-\frac{E(\sigma)}{2\Lambda}})   \right\} = 2\Lambda(1 - e^{ - \frac{\mathcal{E}_{\textrm{min}}}{ 2 \Lambda }} ). 
  \end{align*}
  SInce $\tilde{\gamma}$ solves \eqref{general ET},  we have:
  \begin{equation*}
     \mathcal{E}_{\Lambda,\rm{KL}}(\tilde{\gamma}|\mu,\nu) = \inf_{\sigma\in\mathcal{P}(\mathbb{R}^d\times\mathbb{R}^d)}\min_{M\geq 0} ~\left\{ M\mathcal{E}_{\Lambda,\rm{KL}}(\sigma|\mu,\nu)+2\Lambda (M\log M -M )+2\Lambda \right\} =2\Lambda(1 - e^{-\frac{\mathcal{E}_{ \textrm{  min}}}{2\Lambda}       } ). 
  \end{equation*}
  Now we write $\tilde{\gamma}=Z\gamma$, with $Z=\tilde{\gamma}(\mathbb{R}^d\times\mathbb{R}^d)$, $\gamma\in\mathcal{P}(\mathbb{R}^d \times \mathbb{R}^d)$. We have:
  \begin{equation*}
     Z\mathcal{E}_{\Lambda,\rm{KL}}(\gamma|\mu,\nu)+2\Lambda (Z\log Z-Z)+2\Lambda = 2\Lambda(1- e^{-\frac{\mathcal{E}_{    \textrm{   min }  }   }{2\Lambda}     } )
  \end{equation*}
  However, we have:
  \begin{equation}
     Z\mathcal{E}_{\Lambda,\rm{KL}}(\gamma|\mu,\nu)+2\Lambda (Z\log Z-Z)+2\Lambda  \geq 2\Lambda( 1 - e^{-\frac{\mathcal{E}_{\Lambda,\rm{KL}}(\gamma|\mu,\nu)}{2\Lambda}  }). \label{intermediate inq}
  \end{equation}
  This gives:
  \begin{equation*}
    2\Lambda(1- e^{-\frac{\mathcal{E}_{    \textrm{   min }  }   }{2\Lambda}       }) \geq    2\Lambda(1 - e^{-\frac{\mathcal{E}_{\Lambda,\rm{KL}}(\gamma|\mu,\nu)}{2\Lambda} }) \quad \Rightarrow \quad \mathcal{E}_{\Lambda,\rm{KL}}(\gamma|\mu,\nu)\leq
    \mathcal{E}_{    \textrm{   min }  }. 
  \end{equation*}
  As a result, we have: $\mathcal{E}_{\Lambda,\rm{KL}}(\gamma|\mu,\nu) = \mathcal{E}_{    \textrm{   min }   } $, i.e. $\gamma$ solves problem \eqref{constrained-ET}. And inequality \eqref{intermediate inq} becomes equality, this shows $Z = e^{-\frac{\mathcal{E}_{ \textrm{ min    }} }{2\Lambda}     }$. 
\end{proof}

The following corollary shows the uniqueness of constrained Entropy Transport problem.
\begin{corollary}
  The constrained Entropy Transport problem admits a unique optimal solution.
\end{corollary}

\begin{proof}
We still assume that $\tilde{\gamma}$ and $\gamma$ are solutions to \eqref{general ET} and \eqref{constrained-ET} respectively as stated in Theorem \ref{rel}.  Suppose despite $\gamma$, we have another $\gamma'\in\mathcal{P}(\mathbb{R}^d\times\mathbb{R}^d)$ that also solves \eqref{constrained-ET}
. Set $Z = e^{-\frac{\mathcal{E}_{\textrm{min } }}{2\Lambda}       }$, we can verify that $\mathcal{E}_{\Lambda,\textrm{KL}}(Z\gamma|\mu,\nu) = \mathcal{E}_{\Lambda,\rm{KL}}(Z\gamma'|\mu,\nu)$. This means that $Z\gamma'\neq Z\gamma$ (i.e. $Z\gamma' \neq \tilde{\gamma}$) is another solution to problem \eqref{general ET}. This avoids the uniqueness stated in Theorem \ref{thm existence uniqueness sol }.
\end{proof}

The following theorem characterizes the structure of the optimal solution to problem $\inf_{\gamma\in\mathcal{M}(\mathbb{R}^d\times\mathbb{R}^d)}\left\{\mathcal{E}_{\Lambda,\rm{KL}}(\gamma|\mu,\nu)\right\}$.
\begin{theorem}\label{char optimal solution to general et}
  We assume $\textrm{supp}(\mu)=\textrm{supp}(\nu)=\mathbb{R}^d$. Suppose $\tilde{\gamma}$ is the solution to the Entropy Transport problem:
  \begin{equation}
    \inf_{\gamma\in\mathcal{M}(\mathbb{R}^d\times\mathbb{R}^d)}\left\{\mathcal{E}_{\Lambda,\rm{KL}}(\gamma|\mu,\nu)\right\}.\label{general KL ET pro}
  \end{equation}
  Then there exist certain $\varphi,\psi \in B(\mathbb{R}^d;\mathbb{R})  $ satisfying: 
  $\varphi(x)+\psi(y)\leq c(x,y)$ for any $x,y\in\mathbb{R}^d$ with:
  $\varphi(x)+\psi(y)= c(x,y)$ $\tilde{\gamma}$- almost surely. (Or equivalently, $\tilde{\gamma}$ is concentrated on the set $\{(x,y)|\varphi(x)+\psi(y)=c(x,y)\}$.)
  And:
  \begin{equation*}
     \frac{ d \pi_{1\#}\tilde{\gamma }  }{d\mu} = e^{ - \frac{\varphi}{\Lambda}  }  \quad \frac{d\pi_{2   \# }  \tilde{ \gamma}}{ d \nu} = e^{-\frac{\psi}{\Lambda}}.
  \end{equation*}
 Here $B(\mathbb{R}^d;\mathbb{R})$ denotes the space of Borel functions $f:\mathbb{R}^d\rightarrow \mathbb{R} $.
\end{theorem}

This theorem can be extended to more general cases, see \cite{liero2018optimal}, section 4. Here we give a direct proof  based on the following theorems:
\begin{theorem}[Dual Problem of Entropy Transport problem \eqref{general KL ET pro}]\label{dual}
  Consider the functional:
  \begin{equation}
     \mathcal{D}_{\Lambda,\rm{KL dual}}(u,v|\mu,\nu) = 
     \int\Lambda (1 -e^{-\frac{u}{\Lambda} } )~d\mu + \int \Lambda( 1 - e^{-\frac{v}{\Lambda}} )~d\nu.
     \label{functional et dual}
  \end{equation}
  And the optimization problem (here $\varphi\oplus\psi\leq c$ denotes $\varphi(x)+\psi(y)\leq c(x,y)$ for any $x,y\in\mathbb{R}^d$):
  \begin{equation}
     \sup_{ \substack{ \text{$u,v\in C(\mathbb{R}^d)$} \\ \text{$u\oplus v\leq c$}\\
     } }\left\{\mathcal{D}_{\Lambda,\rm{KL dual}}(u,v|\mu,\nu)\right\}.  \label{ET Dual}
  \end{equation}
  
  Then \eqref{ET Dual} is the dual problem of primal Entropy Transport problem \eqref{general KL ET pro}. We have the strong duality:
  \begin{equation}
     \inf_{\gamma\in\mathcal{P}(\mathbb{R}^d\times\mathbb{R}^d)}\left\{  \mathcal{E}_{\Lambda,\rm{KL}}(\gamma|\mu,\nu)  \right\} = \sup_{u,v\in C(\mathbb{R}^d),~ u\oplus v\leq  c  }\left\{\mathcal{D}_{\Lambda,\rm{KL dual}}(u, v |\mu,\nu)\right\}.   \label{strong- duality}
  \end{equation}
  
\end{theorem}
The proof of Theorem \ref{dual} can be found in \cite{chizat2018unbalanced}, Corollary 5.9. One only need to substitute the cost $c_l(x,y)$ and coefficient $2\delta^2$ in their argument by $\frac{c(x,y)}{\Lambda}$ and $\Lambda$ used in our discussion to get the result.
\begin{theorem}[Existence of dual pair]\label{thm existence dual pair}
There exists dual pairs $(\varphi,\psi)\in B(\mathbb{R}^d;\mathbb{R})\times B(\mathbb{R}^d;\mathbb{R})$ satisfying $\varphi\oplus\psi\leq c$ on $\mathbb{R}^d\times\mathbb{R}^d$, such that:
\begin{equation*}
  \mathcal{D}_{\Lambda,\rm{KL dual}}(\varphi,\psi|\mu,\nu) = \sup_{ u , v\in C(\mathbb{R}^d),~ u\oplus v \leq  c  }\left\{\mathcal{D}_{\Lambda,\rm{KL dual}}(u,v|\mu,\nu)\right\}.
\end{equation*}
\end{theorem}
This result is a special case for the general results on existence of optimal dual pairs (\cite{liero2018optimal}, section 4.4).

\begin{proof}[Proof of Theorem \ref{char optimal solution to general et} ]
Recall $\tilde{\gamma}$ is the optimal solution to \eqref{general KL ET pro} and we denote $\varphi,\psi\in B(\mathbb{R}^d;\mathbb{R})$ as the optimal dual pair stated in Theorem \ref{thm existence dual pair}. 
Now we denote $F(s)=\Lambda(s\log s -s +1)$ and the Legendre transform of $F$ is $F^*(\xi)=\Lambda(e^{\frac{\xi}{\Lambda}}-1)$\footnote{According to the definition of Legendre transform, $F^*(\xi)=\sup_{\xi}\left\{\xi\cdot s-F(s)\right\}=\sup_\xi\left\{\xi\cdot s - \Lambda(s\log s -s +1)\right\}=\Lambda (e^{\frac{\xi}{\Lambda}}-1)$.}. Since $\tilde{\gamma}$ is the optimal solution to problem \eqref{general KL ET pro}, the marginals of $\tilde{\gamma}$ must satisfy $\pi_{1\#}\tilde{\gamma}\ll\mu$, $\pi_{2\#}\tilde{\gamma}\ll\nu$. We then denote $\sigma_1 = \frac{d\pi_{1 \#}\tilde{\gamma}}{d\mu}$ and $\sigma_2 = \frac{d\pi_{2 \#}\tilde{\gamma}}{d\nu}$. We directly verify
\begin{align*}
  \mathcal{E}_{\Lambda,\rm{KL}}(\tilde{\gamma}|\mu,\nu) =& \iint c(x,y)d\tilde{\gamma}+\int F(\sigma_1)d\mu + \int F(\sigma_2)d\nu, \\
  \mathcal{D}_{\Lambda,\rm{KL dual}}(
  \varphi,\psi|\mu,\nu) =&\int -F^*(-\varphi)d\mu + \int-F^*(-\psi)d\nu.
\end{align*}
Since we have (By Theorem \ref{dual}, Theorem \ref{thm existence dual pair}):
\begin{equation*}
  \mathcal{E}_{\Lambda,\rm{KL}}(\tilde{\gamma}|\mu,\nu) = \mathcal{D}_{\Lambda,\rm{KL dual}}(
  \varphi,\psi|\mu,\nu).
\end{equation*}
We will know:
\begin{equation*}
  \iint c(x,y)d\tilde{\gamma} +\int F(\sigma_1)d\mu + \int F(\sigma_2)d\nu = \int -F^*(-\varphi)d\mu + \int-F^*(-\psi)d\nu.
\end{equation*}
This leads to:
\begin{equation}
  \iint (c(x,y)-\varphi(x)-\psi(y))~d\tilde{\gamma} + \int (F(\sigma_1)+F^*(-\varphi)+\varphi\sigma_1)d\mu + \int (F(\sigma_2)+F^*(-\psi)+\psi\sigma_2)d\nu = 0.   \label{prove3intrgls}
\end{equation}
Since $\varphi\oplus\psi\leq c$ we know $c(x,y)-\varphi(x)-\psi(y)\geq 0$; On the other hand, by the definition of Legendre transform, we know $F^*(-\varphi)\geq -\varphi\sigma_1-F(\sigma_1)$, which equivalent to $F(\sigma_1)+F^*(-\varphi)+\varphi\sigma_1\geq 0$. Similarly, $F(\sigma_2)+F^*(-\psi)+\psi\sigma_2\geq 0$. Thus, the three integrals in \eqref{prove3intrgls} are non negative and thus all equal to $0$. The first integral equals $0$ leads to:
\begin{equation*}
  c(x,y)=\varphi(x)+\psi(y) \quad\tilde{\gamma}-\textrm{almost  surely}.
\end{equation*}
The second integral equals $0$ leads to:
\begin{equation*}
 F(\sigma_1(x))+F^*(-\varphi(x))+\varphi(x)\sigma_1(x)=0 \quad \textrm{on}~ \mathbb{R}^d.
\end{equation*}
This gives $\sigma_1(x)=e^{-\frac{\varphi(x)}{\Lambda}}$\footnote{$F(s)+F^*(\xi)-\xi s = 0$ gives $F^*(\xi)=\xi s-F(s)$, i.e. $s = \max_t\{\xi t - F(t)\} = (F')^{-1}(\xi)$. In this case, $s=\sigma_1(x)$, $\xi=-\varphi(x)$, $F(t)=\Lambda(t\log t-t+1)$, thus $\sigma_1(x)=e^{-\frac{\varphi(x)}{\Lambda}}$ } on $\mathbb{R}^d$, this gives $\frac{d\pi_{1 \#}\tilde{\gamma}}{d\mu}=e^{-\frac{\varphi}{\Lambda}}$. Similarly, we can also prove  $\frac{d\pi_{2 \#}\tilde{\gamma}}{d\nu}=e^{-\frac{\psi}{\Lambda}}$.

\end{proof}

The following theorem characterize the structure of optimal distribution of the constrained Entropy Transport problem \eqref{constrained-ET}. It is direct result of Theorem \ref{rel} and Theorem \ref{char optimal solution to general et}:
\begin{theorem}[Characterization of optimal distribution
  $\gamma_{cET}$ to problem \eqref{constrained-ET} ]\label{characterization of optimal gamma}
  Assume $\varphi,\psi$ are the functions mentioned in Theorem \ref{char optimal solution to general et}. Suppose $\gamma_{cET}$ solves the constrained Entropy Transport problem \eqref{constrained-ET}. Then we also have:
  \begin{equation*}
  \varphi(x)+\psi(y)=c(x,y) \quad  \gamma_{c ET} - \textrm{almost surely},
  \end{equation*}
  and
  \begin{equation}
    \pi_{1 \#}\gamma_{c ET}=e^{\frac{   \mathcal{E}_{   \textrm{min } }  - 2\varphi }{2\Lambda}} \mu \quad   \pi_{2 \#}\gamma_{cET} = e^{\frac{   \mathcal{E}_{   \textrm{min } }  - 2\psi }{2\Lambda}}\nu.
    \label{marginal entropy transport on P}
  \end{equation}
\end{theorem}

Recall in Optimal Transport problem \eqref{OT}, we may compare Theorem \ref{characterization of optimal gamma} with the following theorem on Optimal Transport problem \cite{ambrosio2008gradient}:
\begin{theorem}[Characterization of optimal distribution $\gamma_{OT}$ to problem \eqref{OT}]\label{characterization of optimal distribution to Optimal Transport}
If we assume additional condition on the cost function: $c(x,y)\leq a(x)+b(y)$ with $a\in L^1(\mu)$, $b\in L^1(\nu)$. Then there exists an optimal distribution $\gamma_{OT}$ to problem \eqref{OT}. There exist $\varphi,\psi\in C(\mathbb{R}^d)$ such that $\varphi(x)+\psi(y)\leq c(x,y)$ for any $x,y\in\mathbb{R}^d$ with:
\begin{equation*}
  \varphi(x)+\psi(y)=c(x,y) \quad \gamma_{OT}-\textrm{almost surely},
\end{equation*}
and
\begin{equation}
  \pi_{1 \#}\gamma_{OT} = \mu \quad \pi_{2 \#}\gamma_{OT} = \nu. \label{marginal optimal transport on P}
\end{equation}
\end{theorem}
Since we are using constrained Entropy Transport problem \eqref{constrained-ET} to approximate Optimal Transport problem \eqref{OT}, we are interested in comparing the difference between their optimal distributions $\gamma_{c ET}$ and $\gamma_{OT}$. Although we can identify their difference from marginal conditions \eqref{marginal entropy transport on P} and \eqref{marginal optimal transport on P} described in Theorem \ref{characterization of optimal gamma} and \ref{characterization of optimal distribution to Optimal Transport}, currently we do not have a quantitative analysis on the difference between the optimal distributions to problem \eqref{OT} and \eqref{constrained-ET}. This may serve as one of our future research directions.

\subsection{$\Gamma$-convergence property}\label{App a3}

Despite the discussion for a fixed $\Lambda$, we also establish asymptotic results for \eqref{constrained-ET} as $\Lambda \rightarrow +\infty$. We consider $\mathcal{P}_2(\mathbb{R}^d\times\mathbb{R}^d)$ equipped with the topology of weak convergence. We are able to establish the following $\Gamma$-convergence results for the functional $\mathcal{E}_{\Lambda,\rm{KL}}( \cdot|\mu,\nu)$ defined on $\mathcal{P}_2(\mathbb{R}^d\times\mathbb{R}^d)$:
\begin{theorem}[$\Gamma$-convergence]\label{Appen gamma}
Suppose the cost function is quadratic: $c(x,y)=|x-y|^2$. Assuming that we are given $\mu,\nu\in\mathcal{P}_2(\mathbb{R}^d)$ and at least one of $\mu$ and $\nu$ satisfies the Logarithmic  Sobolev inequality with constant $ K  >0$. Let  $\{\Lambda_n\}$ be a positive increasing sequence, satisfying $\lim_{n\rightarrow \infty}\Lambda_n = +\infty$. We consider the sequence of functionals $\{ \mathcal{E}_{\Lambda_n,\rm{KL}}(\cdot|\mu,\nu) \}$. Recall the functional $\mathcal{E}_{\iota}(\cdot|\mu,\nu)$ defined in \eqref{iota ET functional}. Then $\{ \mathcal{E}_{\Lambda_n,\rm{KL}}(\cdot|\mu,\nu)\}$ $\Gamma$- converges to $\mathcal{E}_{\iota}(\cdot|\mu,\nu)$ on $\mathcal{P}_2(\mathbb{R}^d\times\mathbb{R}^d)$.
\end{theorem}
Before we present the proof, we introduce the Logarithmic Sobolev inequality \cite{villani2008optimal}:
\begin{definition}
We say a probability distribution $\mu$ satisfying the Logarithmic Sobolev inequality with constant $K>0$, if for any probability measure $\tilde{\mu}\ll\mu$, we have 
\begin{equation*}
    D_{\rm{KL}}  (\tilde{\mu} \| \mu) \leq \frac{1}{2K} I(\tilde{\mu} | \mu).
\end{equation*}
Here $I(\bar{\mu} | \mu)$ is the Fisher information defined as
\begin{equation*}
  I(\tilde{\mu}|\mu) = \int \left|\nabla\log\left( \frac{d\tilde{\mu}}{d\mu} \right)\right|^2~d\tilde{\mu}.
\end{equation*}
\end{definition}
We also need the following Talagrand inequality \cite{villani2008optimal}:
\begin{theorem}
  Suppose $\mu\in\mathcal{P}_2(\mathbb{R}^m)$ satisfies the Logarithmic Sobolev inequality with constant $ K > 0$. Then $\mu$ also satisfies the following Talagrand inequality: for any $ \tilde{\mu} \in \mathcal{P}_2(\mathbb{R}^m )  $,
  \begin{equation}
    W_2(\tilde{\mu}, \mu)\leq \sqrt{\frac{2D_{\rm{KL}}(\tilde{\mu}\|\mu)}{K}}. \label{Talagrand ineq}
  \end{equation}
\end{theorem}

Now we can prove Theorem \ref{Appen gamma}.
\begin{proof}[Proof of Theorem \ref{Appen gamma}]
 First, we notice that $\mathcal{P}_2(\mathbb{R}^d\times\mathbb{R}^d)$ equipped with the topology of weak convergence is metrizable by the 2-Wasserstein distance \cite{villani2008optimal}. Thus $\mathcal{P}_2(\mathbb{R}^d\times\mathbb{R}^d)$ is metric space and is first countable. For first countable space, we only need to verify the upper bound inequality and the lower bound inequality in order to prove $\Gamma$-convergence.\\
 1) Upper bound inequality:
 For every $\gamma \in \mathcal{P}(\mathbb{R}^d\times\mathbb{R}^d)$, there is a sequence $\{ \gamma_n\}$ converging to $\gamma$ such that
 \begin{equation}
   \limsup _{{n\to \infty }}\mathcal{E}_{\Lambda_n,\rm{KL} }(\gamma_{n}|\mu,\nu) \leq
   \mathcal{E}_{\iota}(\gamma|\mu,\nu). \label{up ine}
 \end{equation}
 We set $\gamma_n = \gamma$ for all $n\geq 1$, now there are two cases:\\
 (a)  If $\gamma$ doesn't satisfy at least one of the marginal constraints, i.e. $\pi_{1\#}\gamma\neq\mu$ or $\pi_{2 \#} \gamma \neq \nu$, then $\mathcal{E}_{\iota}(\gamma|\mu,\nu)=+\infty$ and the inequality \eqref{up ine} definitely holds;\\
 (b)  If $\gamma$ satisfies the marginal constraints, $\pi_{1\#}\gamma = \mu$, $\pi_{2 \#} \gamma =\nu$, then $\mathcal{E}_{\Lambda_n,\rm{KL}}(\gamma|\mu,\nu)=\mathcal{E}_{\iota}(\gamma|\mu,\nu)$, \eqref{up ine} also holds.\\
 2) Lower bound inequality:
 For every sequence $\{\gamma_n\}$ converging to $\gamma$,
 \begin{equation}
  \liminf _{{n\to \infty }}\mathcal{E}_{\Lambda_n, \rm{KL}   }(\gamma_{n}|\mu,\nu)\geq\mathcal{E}_{  \iota  } (\gamma | \mu,\nu ).\label{lower ine}
 \end{equation}
 We still separate our discussion into two cases:\\
 (a) If $\gamma$ satisfies the marginal constraints, we have:
 \begin{align*}
     \liminf_{n\rightarrow \infty}\mathcal{E}_{\Lambda_n, \rm{KL}}(\gamma_n|~\mu,\nu) &= \liminf_{n\rightarrow \infty}\int_{\mathbb{R}^d\times\mathbb{R}^d}c(x,y)d\gamma_n(x,y) +\Lambda_n D_{\rm{KL}}({\pi_1}_{\#}\gamma_n\|\mu)+\Lambda_n D_{\rm{KL}}({\pi_2}_{\#}\gamma_n\|\nu) \\
     &\geq \liminf_{n\rightarrow \infty}\int_{\mathcal{M}\times\mathcal{M}}c(x,y)d\gamma_n(x,y)\\
     &= \int_{\mathbb{R}^d\times\mathbb{R}^d}c(x,y)d\gamma(x,y) \\
     &= \mathcal{E}_{\iota}(\gamma|~\mu,\nu).
 \end{align*}
 Here we use the fact that $D_{\textrm{KL}}(\mu_1\|\mu_2)\geq 0$ for any $\mu_1,\mu_2\in\mathcal{P}(\mathbb{R}^d)$.\\
 (b) If $\gamma$ doesn't satisfy at least one of the marginal constraints, without loss of generality, assume that $W_2({\pi_1}_{\#} \gamma, \mu) = \delta > 0$. We have:
\begin{align*}
    W_2({\pi_1}_{\#} \gamma, \mu) 
    &\leq W_2({\pi_1}_{\#} \gamma, {\pi_1}_{\#} \gamma_n) + W_2( {\pi_1}_{\#} \gamma_n, \mu) \\
    &\leq W_2(\gamma,\gamma_n) + W_2( {\pi_1}_{\#} \gamma_n, \mu). 
\end{align*}
We can choose large enough $N$ such that when $n>N$, $W_2(\gamma, \gamma_n)\le \delta/2$, then we have $W_2( {\pi_1}_{\#} \gamma_n, \mu) \ge \delta/2$.

According to Talagrand inequality \eqref{Talagrand ineq}, we have:
\begin{align*}
   \sqrt{\frac{2 D_{\rm{KL}}({\pi_1}_{\#} \gamma_n \|\mu)}{K}}  &\ge W_2( {\pi_1}_{\#} \gamma_n, \mu) \\
    &\ge  \delta/2,
\end{align*}
i.e., when $n>N$, $D_{\rm{KL}}({\pi_1}_{\#} \gamma_n\|\mu)\geq K\frac{\delta^2}{8}$. This  implies:
\begin{equation*}
\mathcal{E}_{\Lambda_n,\rm{KL}}(\gamma_n|\mu,\nu)\geq \Lambda_n  K \frac{\delta^2}{8}.
\end{equation*}
Therefore we show that:
\begin{align*}
    \liminf_{n\rightarrow \infty}\mathcal{E}_{\Lambda_n,\rm{KL}}(\gamma_n|\mu,\nu)=  +\infty = \mathcal{E}_{\iota} (\gamma|\mu,\nu).
\end{align*}
Thus, combining (a) and (b), we have proved \eqref{lower ine}. And combining \eqref{up ine} \eqref{lower ine}, we have shown that $\{\mathcal{E}_{\Lambda_n,\rm{KL}}(\cdot| \mu,\nu)\}$ $\Gamma$-converges to $\mathcal{E}_{\iota}(\cdot|\mu,\nu).$
\end{proof}

We can then establish the equi-coercive property for the family of functionals $\{\mathcal{E}_{\Lambda_n,\rm{KL}}(\cdot|\mu,\nu)\}_n$. We can apply the Fundamental Theorem of $\Gamma$-convergence\cite{dal2012introduction}
 \cite{braides2006handbook} to establish the following asymptotic result:
\begin{theorem}[Property of $\Gamma$-convergence]\label{appen: property of convergence }
Suppose the cost function is quadratic: $c(x,y)=|x-y|^2$. Assuming $\mu,\nu\in\mathcal{P}_2(\mathbb{R}^d)$ and both $\mu, \nu$ satisfies the Logarithmic Sobolev inequality with constants $ K_\mu,K_\nu>0$. According to Corollary \ref{thm existence uniqueness sol }, the  
problem \eqref{constrained-ET} with functional $\mathcal{E}_{\Lambda_n,\rm{KL}}(\cdot | \mu,\nu   )$
admits a unique optimal solution, let us denote it as $\gamma_n$. According to Theorem \ref{uniqueness optimal distribution OT}, the Optimal Transport problem \eqref{OT} also admits a unique solution, we denote it as $\gamma_{OT}$. Then: $\lim_{n\rightarrow \infty}\gamma_n = \gamma_{OT}$ in $\mathcal{P}_2(\mathbb{R}^d\times \mathbb{R}^d)$. 
\end{theorem}
Before we prove this theorem, we introduce the definition of equi-coerciveness:
\begin{definition}
A family of functions $\{F_n\}$ on $X$ is said to be equi-
coercive, if for every $\alpha \in \mathbb{R}$, there is a compact set $C_\alpha$ of $X$ such that the sublevel sets $\{F_n \leq \alpha \} \subset C_\alpha$ for all $n$.
\end{definition}
To prove Theorem \ref{appen: property of convergence }, we first establish the following two lemmas:
\begin{lemma}\label{equi coercive Lemma}
 Suppose $d_0>0$. Denote
  \begin{equation*}
     C = \{\gamma\in\mathcal{P}_2(\mathbb{R}^d\times\mathbb{R}^d)~ |
      W_2(\pi_{1 \#}\gamma, \mu)\leq d_0,~W_2(\pi_{2 \#}\gamma, \nu)\leq d_0 \}.
  \end{equation*}
  Then $C$ is compact set of $\mathcal{P}_2(\mathbb{R}^d\times\mathbb{R}^d)$. Recall that $\mathcal{P}_2(\mathbb{R}^d\times\mathbb{R}^d)$ is equipped with the topology of weak convergence.
\end{lemma}
\begin{proof}[Proof of the Lemma \ref{equi coercive Lemma}]
 According to Prokhorov's Theorem \cite{billingsley2013convergence}, we only need to show that $C$ is tight. That is: for any $\epsilon>0$, we can find a compact set $E_\epsilon\subset \mathbb{R}^d\times\mathbb{R}^d$, such that
 \begin{equation*}
    \gamma(E_\epsilon)\geq 1-\epsilon \quad \forall~\gamma\in C.
 \end{equation*}
 Let us denote $B_R^{d}\subset \mathbb{R}^d$ as the ball centered at origin with radius $R$ in $\mathbb{R}^d$. Since $\mu,\nu$ are probability measures, for arbitrary $\epsilon>0$, we can pick $R(\mu,\epsilon), R(\nu,\epsilon)>0$ such that
 \begin{equation*}
    \mu(B_{R(\mu,\epsilon)}^d)\geq 1-\epsilon,\quad \nu(B_{R(\nu,\epsilon)}^d)\geq 1-\epsilon.
 \end{equation*}
 Now for any chosen $\epsilon>0$, we choose
 \begin{equation*}
     R = \sqrt{\frac{4d_0^2}{\epsilon}} \quad \textrm{and} \quad \tilde{R}=\sqrt{(R(\mu,\frac{\epsilon}{4})+R)^2+(R(\nu,\frac{\epsilon}{4})+R)^2}.
 \end{equation*}
 Now we prove $\gamma(B_{\tilde{R}}^{2d})\geq 1-\epsilon$ for any $\gamma\in C$:\\
  Denote $\gamma_1=\pi_{1 \#} \gamma$, let $\gamma_{OT}$ be the optimal coupling of $\gamma_1$ and $\mu$, i.e. 
 \begin{equation*}
  \gamma_{OT} = \underset{\pi\in\Pi(\gamma_1,\mu)}{\textrm{argmin}}\left\{\iint c(x,y)~d\pi(x,y)\right\}.
 \end{equation*}
 Then (here, we denote $R_\mu = R(\mu,\frac{\epsilon}{4})$ for short hand):
 \begin{equation*}
   d_0^2 \geq W_2^2(\gamma_1, \mu) = \int_{\mathbb{R}^d}\int_{ \mathbb{R}^d}|x-y|^2d\gamma_{OT}(x,y)\geq \int_{\overline{B_{R_\mu+R}^d}}\int_{ B^d_{R_\mu}} |x-y|^2 ~d\gamma_{OT}(x,y)\geq R^2\int_{\overline{B_{R_\mu+R}^d}}\int_{B^d_{R_\mu}}~d\gamma_{OT}(x,y).
 \end{equation*}
 This gives:
 \begin{equation}
   \int_{\overline{B_{R_\mu+R}^d}}\int_{B_{R_\mu}^d}~d\gamma_{OT}(x,y)\leq\frac{d_0^2}{R^2}=\frac{\epsilon}{4}.  \label{lemmap 1}
 \end{equation}
 On the other hand, one have:
 \begin{equation}
   \int_{\overline{B_{R_\mu+R}^d}}\int_{\overline{B_{R_\mu}^d}} d\gamma_{OT}(x,y) \leq \int_{\overline{B_{R_\mu}^d}} d\mu(y) = 1-\mu(B_{R_\mu}^d) \leq \frac{\epsilon}{4}. \label{lemmap 2}
 \end{equation}
 Now sum \eqref{lemmap 1} and \eqref{lemmap 2} together, we have:
 \begin{equation*}
   \gamma_1\left(\overline{B_{R_\mu+R}^{d}}\right)=\int_{\overline{B_{R_\mu+R}^d}}\int_{\mathbb{R}^d}d\gamma_{OT} = \iint_{\overline{B_{R_\mu+R}^d}\times   B_{R_\mu}^d} ~d\gamma_{OT}(x,y) + \iint_{\overline{B_{R_\mu+R}^d}\times \overline{B_{R_\mu}^d}}~d\gamma_{OT}(x,y) \leq \frac{\epsilon}{2}.
 \end{equation*}
 Similarly, denote $\gamma_2=\pi_{2 \#}\gamma$, 
 we have:
 \begin{equation*}
    \gamma_2\left(\overline{B_{R_\nu+R}^{d}}\right)\leq \frac{\epsilon}{2}.
 \end{equation*}
 As a result, for any $\epsilon>0$, we can pick the compact ball $B_{\tilde{R}}^{2d}\subset \mathbb{R}^d\times\mathbb{R}^d $, so that for any $\gamma\in C$, 
 \begin{equation}
   \gamma(B_{\tilde{R}}^{2d})=1-\gamma\left(\overline{B_{\tilde{R}}^{2d}}\right)\geq 1-\gamma\left(\left(\overline{B_{R_\mu+R}^d}\times\mathbb{R}^d\right)\bigcup\left(\mathbb{R}^d\times\overline{B_{R_\nu+R}^{d}}\right)\right)\geq 1-\gamma_1 \left(\overline{B_{R_\mu+R}^d}\right) - \gamma_2\left(\overline{B_{R_\nu+R}^d}\right)\geq 1-\epsilon,  \label{tight ine}
 \end{equation}
 here we are using the fact:
 \begin{equation*}
   \overline{B_{\tilde{R}}^{2d}}\subset \left(\overline{B_{R_\mu+R}^d}\times\mathbb{R}^d\right)\bigcup\left(\mathbb{R}^d\times\overline{B_{R_\nu+R}^{d}}\right).
 \end{equation*}
 The inequality \eqref{tight ine} proves the tightness of set $C$ and thus $C$ is compact set in $\mathcal{P}_2(\mathbb{R}^d\times\mathbb{R}^d)$.
\end{proof}
\begin{lemma}\label{lemma equi-coercv}
  Assuming $\mu,\nu\in\mathcal{P}_2(\mathbb{R}^d)$ and both $\mu, \nu$ satisfies the Logarithmic Sobolev inequality with constants $ K_\mu,K_\nu>0$. The sequence of functionals $\{\mathcal{E}_{\Lambda_n,\rm{KL}}(\cdot|\mu,\nu)\}$ defined on $\mathcal{P}_2(\mathbb{R}^d\times\mathbb{R}^d)$ with positive increasing sequence $\{\Lambda_n\}$ is equi-coercive.
\end{lemma}
\begin{proof}[proof of Lemma \eqref{lemma equi-coercv}]
By Talagrand inequality \eqref{Talagrand ineq} involving $\mu,\nu$:
\begin{equation*}
  D_{\textrm{KL}}(\rho\|\mu)\geq \frac{K_\mu}{2}W^2_2(\rho,\mu) \quad D_{\textrm{KL}}(\rho\|\nu)\geq \frac{K_\nu}{2} W^2_2(\rho,\nu) \quad \forall~\rho\in\mathcal{P}_2(\mathbb{R}^d).
\end{equation*}
Thus,
\begin{equation*}
  \mathcal{E}_{\Lambda_n,\rm{KL}}(\gamma|\mu,\nu)\geq \Lambda_1 \left(\frac{K_\mu}{ 2 } W_2^2(\gamma_1,\mu) + \frac{K_\nu}{2}W_2^2(\gamma_2,\nu)\right).
\end{equation*}
For any $\alpha\geq 0$, we set $d_0 = \max\{\sqrt{\frac{2\alpha}{K_\mu\Lambda_1}},\sqrt{\frac{2\alpha}{K_\nu\Lambda_1}}\}$, then 
\begin{align*}
   \left\{\gamma~|~\gamma\in\mathcal{P}_2(\mathbb{R}^d\times\mathbb{R}^d),~\mathcal{E}_{\Lambda_n,\rm{KL}}(\gamma|\mu,\nu)\leq \alpha\right\} \subset \left\{\gamma~\Big|~W_2(\gamma_1, \mu)\leq d_0, W_2(\gamma_2, \nu) \leq d_0 \right\} \overset{\textrm{denote as}}{=}C_\alpha.
\end{align*}
By Lemma \ref{equi coercive Lemma}, $C_\alpha$ is compact in $\mathcal{P}_2(\mathbb{R}^d\times\mathbb{R}^d)$ for any $\alpha$ (for $\alpha<0$, we simply get empty set and thus is also compact set). Thus the sequence of functionals $\{\mathcal{E}_{\Lambda_n,\rm{KL}}(\cdot|\mu,\nu)\}$ is equi-coercive.
\end{proof}
Now our proof mainly rely on the following fundamental theorem of $\Gamma$-convergence \cite{dal2012introduction} \cite{braides2006handbook}: 
\begin{theorem} \label{main thm gamma cong}
Let $(X, d)$ be a metric space, let $\left\{F_{\theta_n}\right\}$ with $\theta_n\rightarrow +\infty$ be an equi-coercive sequence of functionals on $X$, assume $\{F_{\theta_n}\}$ $\Gamma$-converge to the functional $F$ defined on $X$; Then
\begin{equation*}
\exists \min_X F = \lim_{n\rightarrow\infty} \inf_X F_{\theta_n}.
\end{equation*}
Moreover, if $\{x_n\}$ is a precompact sequence such that $x_n$ is the minimizer of $F_{\theta_n}$: $F_{\theta_n}(x_n)=\inf_X F_{\theta_n}$, then every
limit of a subsequence of $\left\{x_n\right\}$ is a minimum point for $F$.
\end{theorem}
We can now prove Theorem \ref{appen: property of convergence }.
\begin{proof} 
  We apply Theorem \ref{main thm gamma cong} to the sequence of functionals $\{\mathcal{E}_{\Lambda_n,\rm{KL}}(\cdot|\mu,\nu)\}_n$ defined on probability space equipped with 2-Wasserstein metric $(\mathcal{P}_2(\mathbb{R}^d\times\mathbb{R}^d), W_2)$, by Lemma \ref{lemma equi-coercv} , we know that $\{\mathcal{E}_{\Lambda_n,\rm{KL}}(\cdot|\mu,\nu)\}_n$ is equi-coercive. And by Theorem \ref{Appen gamma}, $\{\mathcal{E}_{Lambda_n,\rm{KL}}\}(\cdot|\mu,\nu)\}_n$ $\Gamma$-converge to $\mathcal{E}_\iota(\cdot|\mu,\nu)$. Recall $\gamma_n$ is the unique minimizer of $\mathcal{E}_{\Lambda_n,\rm{KL}}(\cdot|\mu,\nu)$, we are going to show that $\{\gamma_n\}$ is precompact sequence in $\mathcal{P}_2(\mathbb{R}^d\times\mathbb{R}^d)$: We define
  \begin{equation*}
     \alpha = \iint c(x,y)~d(\mu\otimes\nu)= \mathcal{E}_{\Lambda_n,\rm{KL}}(\mu\otimes\nu|\mu,\nu) \quad\forall~n\geq 1.
  \end{equation*}
  Then we have $\mathcal{E}_{\Lambda_n,\rm{KL}}(\gamma_n|\mu,\nu)\leq\mathcal{E}_{\Lambda_n,\rm{KL}}(\mu\otimes\nu|\mu,\nu)=\alpha$ for all $n$, thus
  \begin{equation*}
     \gamma_n\in\left\{\gamma~|~\gamma\in\mathcal{P}_2(\mathbb{R}^d\times\mathbb{R}^d),~\mathcal{E}_{\Lambda_n,\rm{KL}}(\gamma|\mu,\nu)\leq\alpha\right\} \quad \forall~ n\geq1. 
  \end{equation*}
  Now since $\{\mathcal{E}_{\Lambda_n,\rm{KL}}(\cdot|\mu,\nu)\}$ is equi-coercive, we can pick compact $C_\alpha$ such that:
  \begin{equation*}
    \left\{\gamma~|~\gamma\in\mathcal{P}_2(\mathbb{R}^d\times\mathbb{R}^d),~\mathcal{E}_{\Lambda_n,\rm{KL}}(\gamma|\mu,\nu)\leq\alpha\right\}\subset C_\alpha \quad\forall ~n \geq 1.
  \end{equation*}
  Thus all $\{ \gamma_n \}$ lie in the compact set $C_\alpha$ and $\{\gamma_n\}$ is precompact.\\
  Now Theorem \ref{appen: property of convergence } asserts that any limit point of $\{\gamma_n\}$ is a minimum point of $\mathcal{E}_{\iota}(\cdot|\mu,\nu)$, however, $\mathcal{E}_{\iota}(\cdot,\mu,\nu)$ admits unique minimizer $\gamma_{OT}$, we have proved  $\lim_{n\rightarrow\infty}\gamma_n=\gamma_{OT}$.
\end{proof}

\newpage

\section{Appendix B}\label{App b}

In this Appendix, we first introduce the basic knowledge of Waserstein manifold and derive the formula for Wasserstein gradient flow for general functionals defined on that manifold. We then give a detailed derivation of the equation for gradient flow \eqref{PDE vers} of functional $\mathcal{E}_{\Lambda,\rm{KL}}(\cdot|\mu,\nu)$.
\subsection{Wasserstein geometry and Wasserstein gradient flows}\label{App b1}

\subsubsection{Wasserstein manifold-like structure}\label{wass_mfld}
Denote the probability space supported on $\mathbb{R}^d$ with densities and finite second order momentum as:
\begin{equation*}
    \mathcal{P}_2(\mathbb{R}^d)=\left\{\gamma~ \Big|   
    \gamma\in\mathcal{P}(\mathbb{R}^d),\gamma\ll\mathscr{L}^d,~\int |x|^2~d\gamma <\infty\right\}.
\end{equation*}
We define the so-called Wasserstein distance (also known as $L^2$-Wasserstein distance) on $\mathcal{P}$ as \cite{villani2008optimal}:
\begin{equation}
  W_2(\gamma_1,\gamma_2)= \left(\inf_{\pi\in\Pi(\gamma_1,\gamma_2)}\iint |x-y|^2 ~d\pi(x,y)\right)^{1/2}. \label{def_wass_dist}
\end{equation}
Here $\Pi(\gamma_1,\gamma_2)$ is the set of joint distributions on $\mathbb{R}^d\times \mathbb{R}^d$ with fixed marginal distributions as $\gamma_1, \gamma_2$ (recall definition \eqref{eq:Pi}).
If we treat $\mathcal{P}( \mathbb{R}^d )$ as an infinite dimensional manifold,  then the Wasserstein distance $W_2$ can induce a metric $g^W$ on the tangent bundle $T\mathcal{P}(\mathbb{R}^d)$ and then $\mathcal{P}(\mathbb{R}^d)$ becomes a Riemmanian manifold. We now directly give the definition of $g^W$ and then prove the equivalence between $g^W$ and $W_2$: One can identify the tangent space at $\gamma$ as:
\begin{equation*}
T_\gamma\mathcal{P}=\left\{\dot\gamma ~\Big|~\dot{\gamma} ~\textrm{is a signed measure},~ \int
 d\dot\gamma=0
\right\}.
\end{equation*}
Now for a specific $\gamma\in\mathcal{P}( \mathbb{R}^d )$ and $\dot\gamma_i\in T_\gamma\mathcal{P}( \mathbb{R}^d )$, $i=1,2$, we define the Wasserstein metric tensor $g^W$ as: \cite{Lafferty,otto2001}
\begin{equation}
g^W(\gamma)(\dot{\gamma}_1,\dot{\gamma}_2)=\int  \nabla\psi_1(x)\cdot\nabla\psi_2(x)\gamma(x) ~dx,\label{def_metric}
\end{equation}
where $\psi_1,\psi_2$ satisfies:\footnote{$\psi_i$, $i=1,2$ satisfy the equation in the weak sense that:
\begin{equation*}
  \int f~d\dot\gamma = \int \nabla f\cdot\nabla\psi_i~d\gamma \quad\forall f \in C^{\infty}_0(\mathbb{R}^d) \quad i=1,2.
\end{equation*}
}
\begin{equation}
\dot{\gamma_i}=-\nabla\cdot(\gamma_i\nabla\psi_i) \quad i=1,2,\label{hodge}
\end{equation}
with boundary conditions
$$\lim_{x\rightarrow +\infty }   \frac{d\gamma_1}{d\mathscr{L}^d}(x)\nabla\psi_i(x)=0 \quad i=1,2.$$
according to the above definition, we can write:
\begin{equation*}
    g^W(\gamma)(\dot\gamma_1,\dot\gamma_2) = \int\psi_1(-\nabla\cdot(\gamma\nabla\psi_2))~dx = \int (-\nabla\cdot(\gamma\nabla))^{-1}(\dot\gamma_1) 
    \dot\gamma_2~dx.
\end{equation*}
Thus, we can identify $g^W(\gamma)$ as $(-\nabla\cdot(\gamma\nabla))^{-1}$. When $\textrm{supp}(\gamma)=\mathbb{R}^d$, $g^W(\gamma)$ is a positive definite bilinear form defined on tangent bundle $T\mathcal{P}(\mathbb{R}^d)=\{(\gamma,\dot\gamma)~|~ \gamma\in \mathcal{P}(\mathbb{R}^d),~\dot\gamma\in T_\gamma\mathcal{P}(\mathbb{R}^d ) \}$ and we can treat $\mathcal{P}$ as a Riemannian manifold and we will call the manifold $(\mathcal{P}(\mathbb{R}^d),g^W)$ Wasserstein manifold-like structure \cite{otto2001}.

\subsubsection{Wasserstein gradient}\label{background wass_grad}
We denote the Wasserstein gradient $\textrm{grad}_W$ as manifold gradient on $(\mathcal{P}(\mathbb{R}^d),g^W)$.
In Riemannian geometry, the manifold gradient should be compatible with the metric, which implies that for any smooth $\mathcal{F}$ defined on $\mathcal{P}$ and for any $\gamma\in\mathcal{P}(\mathbb{R}^d) $, consider arbitrary differentiable curve $\{\gamma_t\}_{t\in(-\delta,\delta)}$ with $\gamma_0=\gamma$, we always have:
\begin{equation*}
 \frac{d}{dt}\mathcal{F}(\gamma_t)\Big\vert_{t=0} = g^W(\gamma)(\textrm{grad}_W\mathcal{F}(\gamma) , ~ \dot\gamma_0).
\end{equation*}
Since we can write: 
\[ \frac{d}{dt}\mathcal{F}(\gamma_t)\Big\vert_{t=0}=\int \frac{\delta \mathcal{F}(\gamma)}{\delta\gamma}(x)   d\dot\gamma_0 = \left\langle \frac{\delta\mathcal{F}(\gamma)}{\delta\gamma} , \dot\gamma_0\right\rangle,\]
here $ \frac{\delta\mathcal{F}(\gamma)}{\delta\gamma}(x)$ is the functional derivative of $\mathcal{F}$ at point $x\in\mathbb{R}^d$, we then have:
\begin{equation*}
 \left\langle \frac{\delta\mathcal{F}(\gamma)}{\delta\gamma},\dot\gamma_0\right\rangle = g^W(\gamma)(\textrm{grad}_W\mathcal{F}(\gamma) , ~ \dot\gamma_0)\quad \forall~\dot\gamma_0\in T_\gamma\mathcal{P}(\mathbb{ R}^d).
\end{equation*}
This leads to the following useful formula for computing Wasserstein gradient of functional $\mathcal{F}$:
\begin{equation}
\begin{split}
\textrm{grad}_W\mathcal{F}(\gamma)=&{g^{W}(\gamma)}^{-1}\left(\frac{\delta\mathcal{F}(\gamma)}{\delta\gamma}\right)\\
  =&-\nabla\cdot\left(\gamma \nabla~ \frac{\delta\mathcal{F}(\gamma)}{\delta\gamma }\right),
   \end{split}
   \label{gradflow}
\end{equation}

Thus, we can formulate the Wasserstein gradient flow of functional $\mathcal{F}$ as:
\begin{equation}
  \frac{\partial \gamma }{\partial t} = -\textrm{grad}_W \mathcal{F}(\gamma)\quad   \Longleftrightarrow\quad\frac{\partial \gamma }{\partial t} = \nabla\cdot\left(\gamma\nabla\frac{\delta\mathcal{F}(\gamma)}{\delta\gamma}\right).  \label{wasserstein grad flow measure}
\end{equation}
We can also formulate Wasserstein gradient flow of $\mathcal{F}$ as an equation of density function $\rho$ of $\gamma$:
\begin{equation}
  \frac{\partial \gamma }{\partial t} = -\textrm{grad}_W \mathcal{F}(\gamma)\quad   \Longleftrightarrow\quad\frac{\partial\rho}{\partial t} = \nabla\cdot\left(\rho\nabla\frac{\delta F(\rho)}{\delta\rho}\right). \label{wasserstein grad flow density}
\end{equation}
Here $F(\rho)=\mathcal{F}(\rho\mathscr{L}^d)$ and $\frac{\delta F( \rho)}{\delta\rho}$ is the functional derivative of functional $F$ at density function $\rho$.

\subsection{Derivation of Wasserstein gradient flow for Entropy Transport Functional}\label{App b2}

We now follow the previous section to compute the gradient flow of $\mathcal{E}_{\Lambda,\rm{KL}}(\cdot|\mu,\nu)$ on $\mathcal{P}_2(\mathbb{R}^d\times\mathbb{R}^d)$. We assume every thing is in the form of Radon-Nikodym Derivative, i.e. we assume $\rho = \frac{d\gamma}{d\mathscr{L}^{2d}}$ and $\varrho_1=\frac{d\mu}{d\mathscr{L}^d}$, $\varrho_2=\frac{d\nu}{d\mathscr{L}^d}$. We denote $\rho_1=\frac{d{\pi_1}_{\#}\gamma}{d\mathscr{L}^d},\rho_2=\frac{d{\pi_2}_{\#}\gamma}{d\mathscr{L}^d}$, then $\rho_1 = \int \rho ~dy$, $\rho_2 = \int \rho ~dx$. We write the functional $\mathcal{E}_{\Lambda,\rm{KL}}(\gamma|\mu,\nu)$ as $E(\rho)$ for shorthand, then:
\begin{equation*}
   E(\rho)=\iint_{\mathbb{R}^d\times \mathbb{R}^d } \left(c(x,y)+\Lambda\log\left(\frac{\rho_1(x)}{\varrho_1(x)}\right)+\Lambda\log\left(\frac{\rho_2(y)}{\varrho_2(y)}\right)\right)\rho(x,y)~dxdy.
\end{equation*}
To compute $L^2$ variation of $E$ 
, suppose $\rho>0$ and consider arbitrary $\sigma\in C_0(\mathbb{R}^d\times\mathbb{R}^d)$. We denote $\sigma_1(x)=\int\sigma(x,y)~dy$ and $\sigma_2(y)=\int\sigma(x,y)~dx$. We now compute $\frac{d}{d h}
E(\rho+h\sigma)\Big|_{h=0}$ as:
\begin{align*}
   & \frac{d}{dh}E(\rho+h\sigma)\Big|_{h=0}\\
 = & \frac{d}{dh}\left[\iint \left( c(x,y)+\Lambda\log\left(\frac{\rho_1(x)+h\sigma_1(x)}{\varrho_1(x)}\right) + \Lambda\log\left(\frac{\rho_2(y)+h\sigma_2(y)}{\varrho_2(y)}\right) \right)(\rho(x,y)+h\sigma(x,y))~dxdy \right]_{h=0} \\
 = & \iint \left(\Lambda\frac{\sigma_1(x)}{\rho_1(x)}+\Lambda\frac{\sigma_2(y)}{\rho_2(y)}\right)\rho(x,y)~dxdy+\iint \left(c(x,y) +\Lambda\log\left(\frac{\rho_1(x)}{\varrho_1(x)}\right)+\Lambda\log\left(\frac{\rho_2(y)}{\varrho_2(y)}\right)\right) \sigma(x,y) ~dxdy\\
 =&\int \Lambda \sigma_1(x)~dx + \int \Lambda \sigma_2(y)~dy+\iint\left( c(x,y)+\Lambda\log\left(\frac{\rho_1(x)}{\varrho_1(x)}\right) + \Lambda\log\left(\frac{\rho_2(y)}{\varrho_2(y)}\right)\right)\sigma(x,y)~dxdy\\
 =&\iint \left(2\Lambda+c(x,y)+\Lambda\log\left(\frac{\rho_1(x)}{\varrho_1(x)}\right) + \Lambda\log\left(\frac{\rho_2(y)}{\varrho_2(y)}\right)\right)\sigma(x,y)~dxdy.
\end{align*}
Since $\frac{dE(\rho+h\sigma)}{dh}\Big|_{h=0}=\langle \frac{\delta E(\rho)}{\delta\rho}, \sigma \rangle$, we can thus identify that:
\begin{equation*}
  \frac{\delta E(\rho)}{\delta\rho}=2\Lambda+c(x,y)+\Lambda\log\left(\frac{\rho_1(x)}{\varrho_1(x)}\right) + \Lambda\log\left(\frac{\rho_2(y)}{\varrho_2(y)}\right).
\end{equation*}
Thus, plugging this result into formula \eqref{wasserstein grad flow density}, one can derive:
\begin{equation*}
  \frac{\partial\rho(x,y,t)}{\partial t} = \nabla\cdot\left( \rho(x,y,t) \nabla \left(2\Lambda+c(x,y)+\Lambda\log\left(\frac{\rho_1(x,t)}{\varrho_1(x)}\right) + \Lambda\log\left(\frac{\rho_2(y,t)}{\varrho_2(y)}\right)\right)  \right).
\end{equation*}
Notice that $\nabla$ means gradient with respect to both variables $x$ and $y$, i.e. $\nabla f = \left(\begin{array}{c}
     \nabla_x f \\
      \nabla_yf
\end{array}\right)$ for function $f:\mathbb{R}^d\times\mathbb{R}^d\rightarrow\mathbb{R}$, and $\nabla\cdot \vec{v} = \nabla_x\cdot\vec{v}_1 + \nabla_y\cdot\vec{v}_2$ for vector field $\vec{v} = \left(\begin{array}{c}
  \vec{v}_1\\
  \vec{v}_2
\end{array}\right):\mathbb{R}^d\times\mathbb{R}^d\rightarrow\mathbb{R}^d\times\mathbb{R}^d$, $\vec{v}_1:\mathbb{R}^d\times\mathbb{R}^d\rightarrow\mathbb{R}^d$; $\vec{v}_2:\mathbb{R}^d\times\mathbb{R}^d\rightarrow\mathbb{R}^d$.

Then this equation will simplify to:
\begin{equation*}
  \frac{\partial\rho(x,y,t)}{\partial t} = \nabla\cdot\left( \rho(x,y,t) \nabla \left(  c(x,y)+\Lambda\log\left(\frac{\rho_1(x,t)}{\varrho_1(x)}\right) + \Lambda\log\left(\frac{\rho_2(y,t)}{\varrho_2(y)}\right)\right)  \right).
\end{equation*}
Which is exactly equation \eqref{PDE vers} shown in section \ref{sec3.2}.

\end{document}